\documentclass[12pt]{article}

\usepackage[utf8]{inputenc}
\usepackage[T2A]{fontenc}
\usepackage[english,greek,russian]{babel}
\usepackage[colorlinks=true,linkcolor=black,pagecolor=black,citecolor=black]{hyperref}
\usepackage{amsmath}
\usepackage{amssymb}
\usepackage{amsthm}
\usepackage{array}
\usepackage{bm}
\usepackage{cite}
\usepackage{color}
\usepackage{dsfont}
\usepackage{enumerate}
\usepackage{epsf}
\usepackage{euscript}
\usepackage{graphicx} 
\usepackage{indentfirst}
\usepackage{latexsym}
\usepackage{mathrsfs}
\usepackage{mathtools}
\usepackage{ulem}  
\usepackage{wasysym}
\usepackage{wrapfig}
\usepackage{yfonts} 
\usepackage{tikz}
\usepackage{pgf}

\newtheorem{theorem}{\indent\bf Theorem}[section]

\definecolor{mathcolor}{RGB}{0,75,105}
\definecolor{argcolor} {RGB}{0,105,50}
\definecolor{speccolor}{RGB}{0,40,125}
\newcommand{\dismath} [1] {{\color{mathcolor}#1}}
\newcommand{\distext} [1] {{\color{speccolor}#1}}
\newcommand{\argument}[1] {{\color{argcolor!50!black}{#1}}}

\newcommand{\defnotion}       [1]       {\distext{\textit{#1\/}}}

\newcommand{\remember}        [1]       {}

\newcommand{\bwedge}   {\boldsymbol{\wedge}}

\newcommand{\set}             [1]       {\dismath{\left\{ \argument{#1} \right\}}}

\newcommand{\otuple}          [1]       {\dismath{\langle\argument{#1}\rangle}}

\newcommand{\logic}           [1]       {\dismath{\mathbf{#1}}}         
\newcommand{\lang}            [1]       {\dismath{\mathcal{#1}}}        
\newcommand{\cclass}          [1]       {\dismath{\mathrm{#1}}}         
  \newcommand{\ccls}[1]{\cclass{#1}}

\newcommand{\numbers}         [1]       {\dismath{\mathds{#1}}}
\newcommand{\numN}                      {\numbers{N}}

\newcommand{\equivalence}               {\leftrightarrow}
  \newcommand{\lra} {\equivalence}

\renewcommand{\iff}                     {\mathrel{\dismath{\Longleftrightarrow}}}      
\newcommand{\bydef}                     {\mathrel{\dismath{\leftrightharpoons}}}

\makeatletter
\newcommand{\ExtDiamond}{\@ifnextchar({\ExtDiamond@i}{\ExtDiamond@i({p})}}
\def\ExtDiamond@i(#1){{\Diamond\!\!\!\!\Diamond}_{#1}}
\makeatother

\usepackage[auth-sc]{authblk} 
\usepackage{enumitem}

\usepackage{amsmath}
\usepackage{algorithm}     
\usepackage[noend]{algpseudocode} 
\makeatletter
\renewcommand{\fnum@algorithm}{Algorithm \thealgorithm:} 
\makeatother
\makeatletter
\@addtoreset{algorithm}{section}
\makeatother
\renewcommand{\thealgorithm}{\thesection.\arabic{algorithm}}

\usepackage[a4paper, 
            textwidth=6.5in, 
            left=1in, 
            right=1in, 
            top=1.0in, 
            bottom=1.0in]{geometry} 

\SetEnumitemKey{midsep}{itemsep=0pt}
\setlist{midsep}

\newcounter{\theequation}[section]
\renewcommand{\theequation}{\thesection.\arabic{equation}}

\renewcommand{\dismath} [1] {#1}
\renewcommand{\distext} [1] {#1}
\renewcommand{\argument}[1] {#1} 

\newtheorem{imtheorem}{\bf Theorem}[section]
\newtheorem{imcorollary}[imtheorem]{\bf Corollary}
\newtheorem{imlemma}[imtheorem]{\bf Lemma}

\newtheorem{improposition}[imtheorem]{\bf Proposition}

\newcounter{\theequation}[section]
\renewcommand{\theequation}{\thesection.\arabic{equation}}

\newcommand{\rhu}{\rightharpoonup}
\newcommand{\brhu}{\boldsymbol{\rhu}}

\newcommand{\MR}[1]{#1} 
\newcommand{\insMR}[1]{\MR{#1}}
\newcommand{\delMR}[1]{} 

\usepackage{tocloft}
\begin{document}
\selectlanguage{english}
\sloppy

\title{Semantics for the minimal well-determined logic}
\author[1]{Igor~Gorbunov}
\author[2]{Mikhail~Rybakov} 
\affil[1]{HSE University and Tver State University}
\affil[2]{Higher School of Modern Mathematics MIPT and HSE University} 
\date{} 

\maketitle

\begin{abstract}
The minimal well-determined logic in the language with conjunction and implication is investigated. A calculus for this logic, in which the modus ponens rule is not postulated, is proposed. The main result consists in constructing a semantics for this logic: it is formed by the class of lower semilattices with a greatest element, where the implication is interpreted using a partial function defined via the partial order of the semilattice. This extension of the notion of interpreting logical connectives in a matrix allows for the correct determination of the truth of formulas in the language with conjunction and implication. Soundness and completeness theorems are proved. The proposed semantics creates an opportunity to investigate questions of finite model property for such systems and can also serve as a basis for studying other properties of both the minimal well-determined logic itself and its extensions. As an application of the obtained results, we prove that the set of tautologies of the minimal well-determined logic is decidable in polynomial time and present a corresponding decision algorithm.\end{abstract} 


\tableofcontents

\pagebreak[3]

\section{Introduction}
\setcounter{equation}{0}
\label{sec:1}

The notion of \defnotion{well-determined logics} was introduced by Ryszard W\'{o}jcicki in his work~\cite[Chapter~3]{IG:MR:1}.
Considering classical logic~\cite[1.6.6]{IG:MR:2}, he drew attention to the correspondence between certain tautologies and the relation of logical consequence: if $\logic{Cl}$ is the set of tautologies of classical logic, then, for all formulas $\alpha, \alpha_1,\ldots,\alpha_n$,
$$
\begin{array}{lcl}
\alpha_1\wedge \ldots \wedge\alpha_n\to\alpha\in\logic{Cl}
  & \iff 
  & \alpha_1, \ldots,\alpha_n\vdash\alpha.
\end{array}
$$
Based on formulas of the form $\alpha_1\wedge \ldots \wedge\alpha_n\to\alpha$, he defined a consequence extension operator $\angle K_{\mathit{Cn}}$, setting for any set of formulas $X$ and any formula~$\alpha$
$$
\begin{array}{lcl}
\label{eq:4}
\alpha\in\angle K_{\mathit{Cn}}(X)
  & \iff
  & \parbox[t]{205pt}{there exist $\alpha_1,\ldots,\alpha_n\in X\cup\logic{Cl}$ such that $\alpha_1\wedge \ldots \wedge\alpha_n\to\alpha\in\logic{Cl}$.}
\end{array}
$$

One of the important correspondences between tautologies and the consequence relation is known as the deduction theorem. If $C$ is a consequence operation, then the deduction theorem can be expressed as follows: for any set of formulas $X$ and any formulas $\alpha$ and~$\beta$,
$$
\begin{array}{lcl}
\alpha\in C(X, \beta) 
  & \iff 
  & \beta\to\alpha\in C(X).
\end{array}
$$
However, \cite[1.6.2]{IG:MR:2} notes that to satisfy this correspondence it is sufficient that the \defnotion{weak deduction theorem} (the \defnotion{weak deduction property}) holds: for any finite \insMR{nonempty} set of formulas $\Gamma$ and any formula~$\alpha$,
$$
\begin{array}{lcl}
\alpha\in C(\Gamma)
  & \iff 
  & \mbox{${\bigwedge}\Gamma\to\alpha\in C(\varnothing)$,}
\end{array}
$$
where ${\bigwedge}\Gamma$ is a conjunction of all formulas from $\Gamma$ (one of the many possible).

This fact allows one to generalize the described construction to arbitrary sets of formulas in languages with conjunction and implication, which was done in~\cite[Chapter~3]{IG:MR:1} and~\cite[2.10]{IG:MR:2}. Namely, to each set of formulas $L$ there corresponds an operation $\vec L$ defined on sets of formulas:
\begin{align}
\label{eq:7}
\alpha\in \vec L(X)
  & ~~\iff~~ 
  \parbox[t]{200pt}{there exist $\alpha_1,\ldots,\alpha_n\in X\cup L$ such that $\alpha_1\wedge \ldots \wedge\alpha_n\to\alpha\in L$.}
\end{align}
In the same works, R.\,W\'{o}jcicki proposed a criterion that a set $L$ must satisfy in order for the operation $\vec L$ to be a standard consequence operation satisfying the equality $L=\vec L(\varnothing)$. Logics obtained by such operations possess the weak deduction property and were called \defnotion{well-determined}.

The presence of the weak deduction property leads, in particular, to the fact that to specify a logic one does not require a semantics with strong completeness, since ordinary completeness is sufficient. Indeed, let $\mathbb{S}$ be some characteristic semantics for the set $C(\varnothing)$; let $\mathbb{S}\models\Gamma$ mean that every formula from the set $\Gamma$ is true in~$\mathbb{S}$. Then
$$
\begin{array}{lcl}
\Gamma\vdash_C\alpha
  & \iff 
  & \mathbb{S}\models \{\gamma\to\alpha : \text{$\gamma$ is a conjunction of all formulas from $\Gamma$}\}.
\end{array}
$$

In this work, we continue research in this area. Our investigations are devoted to the least well-determined logic in the language with two binary connectives $\wedge$ and $\to$, which we call \defnotion{conjunction} and \defnotion{implication}. Using the criterion for the deductiveness of a set of formulas (Section~\ref{sec:3}) we present a calculus that defines the least well-determined logic in this language (Section~\ref{sec:4}). Then we consider the matrix semantics of this logic (Sections~\ref{sec:5} and~\ref{sec:6}). Note that, unlike standard matrix semantics, we need an extension of the notion of a valuation, since in the proposed semantics there is no operation corresponding to implication (Section~\ref{sec:6}). Next (Section~\ref{sec:7}), we give a decision algorithm for the set of tautologies of the minimal well-determined logic and estimate the complexity of the decision problem for this set. It was surprising for the authors to find that, despite quite complicated semantics, this problem is decidable in polynomial time: usually ``natural'' logics contain $\logic{Cl}$ (for example, as the image of some embedding) and therefore are $\ccls{coNP}$-hard. Finally (Section~\ref{sec:8}), we discuss the results and related questions.

Where possible, we offer a general exposition of the aspects of the questions under consideration, i.e., we speak about arbitrary logics, not only about the least well-determined logic; this allows us to describe a broader context. Where this is impossible (or, more precisely, where we do not know if it is possible), we describe constructions and present results pertaining directly to the least well-determined logic and not touching upon other logics.

\section{Basic notions and notation}
   \setcounter{equation}{0}
   \label{sec:2}

Let $\Pi=\{p_i:i\in\numN\}$ be a countable set of
\defnotion{propositional variables}, $\Sigma$ a set of finitary \defnotion{function symbols}, which we call \defnotion{logical connectives}, and $\Upsilon=\{(\}\cup\{)\}$ a set of \defnotion{auxiliary symbols}.
The triple $\lang{S}=\langle \Pi, \Sigma, \Upsilon\rangle$ is called a \defnotion{propositional alphabet}.
Any term constructed from the symbols of the alphabet $\lang{S}$ is called an \defnotion{$\lang{S}$\nobreakdash-formula};\footnote{We do not give an explicit definition of a term, as it is standard, and we assume the reader can reconstruct it.} $\lang{S}$\nobreakdash-formulas containing no logical connectives are called \defnotion{atomic}.
Let $\lang{L}_{\lang{S}}$ be the \defnotion{language} consisting of all ${\lang{S}}$-formulas; formulas of the language $\lang{L}_{\lang{S}}$, i.e., $\lang{S}$-formulas, are also called \defnotion{$\lang{L}_{\lang{S}}$\nobreakdash-formulas}.

A \defnotion{substitution} is a homomorphic extension of a mapping $\varepsilon:\Pi\to\lang{L}_{\lang{S}}$ to the set of all $\lang{S}$-formulas. Since such an extension is unique, we denote it also by $\varepsilon$. By $\lang{E}_{\!\lang{S}}$ we denote the set of all such substitutions. For any set $\Gamma$ of $\lang{S}$-formulas, by $\varepsilon\Gamma$ we denote the result of applying the substitution $\varepsilon$ to all formulas from~$\Gamma$.
\insMR{A set of $\lang{S}$-formulas closed under all substitutions is called \defnotion{invariant}.}

Let $X$ be an arbitrary set. We use the following notation:
\begin{itemize}[noitemsep, topsep=5pt, parsep=1pt]
\item $|X|$ is the cardinality of the set $X$;
\item $\mathcal{P}(X) = \set{Y : Y\subseteq X}$;
\item $\mathcal{P}_{\mathit{fin}}(X) = \set{Y\in \mathcal{P}(X) : |Y|\in \numN}$; 
\item $\mathcal{P}_{\mathit{fin}}^+(X) = \mathcal{P}_{\mathit{fin}}^{\phantom{+}}(X)\setminus\set{\varnothing}$,
\end{itemize}
i.e., $\mathcal{P}(X)$, $\mathcal{P}_{\mathit{fin}}^{\phantom{+}}(X)$, and $\mathcal{P}_{\mathit{fin}}^+(X)$ are, respectively, the set of all subsets, the set of all finite subsets, and the set of all nonempty finite subsets of $X$.

A function $C\colon \mathcal{P}(\lang{L}_{\lang{S}})\to\mathcal{P}(\lang{L}_{\lang{S}})$ is called a \defnotion{consequence operation}, or simply a \defnotion{consequence}, if it satisfies the following conditions:
\begin{itemize}[leftmargin=4em, noitemsep, topsep=5pt, parsep=1pt] 
\item[(A1)] $\Gamma\subseteq C(\Gamma)$ \hfill (extensiveness);\mbox{\hspace{2em}}
\item[(A2)] $\Gamma\subseteq\Delta\Rightarrow C(\Gamma)\subseteq C(\Delta)$ \hfill (monotonicity);\mbox{\hspace{2em}}
\item[(A3)] $C(C(\Gamma))=C(\Gamma)$ \hfill (idempotence),\mbox{\hspace{2em}}
\end{itemize}
where $\Gamma$ is any set of $\lang{S}$-formulas.
A consequence $C$ is called \defnotion{standard} if it additionally satisfies two more conditions:
\begin{itemize}[leftmargin=4em, noitemsep, topsep=5pt, parsep=1pt]
\item[(A4)] $\forall\varepsilon\in \lang{E}_{\!\lang{S}}~ \varepsilon C(\Gamma)\subseteq C(\varepsilon\Gamma)$ \hfill (structurality);\mbox{\hspace{2em}}
\item[(A5)] $\varphi\in C(\Gamma)\Rightarrow \exists\Delta\in\mathcal{P}_{\mathit{fin}}(\Gamma)~\varphi\in C(\Delta)$ \hfill (finitariness).\mbox{\hspace{2em}}
\end{itemize}

A \defnotion{logic} is a pair $\bm{C}=\otuple{\lang{S},C}$, where $\lang{S}$ is a propositional alphabet and $C$ is a consequence.
A logic with a standard consequence is called \defnotion{standard}.

Let $\bm{C}=\otuple{\lang{S},C}$ be a logic and $X$ a set of $\lang{S}$-formulas. The set $C(X)=\set{C(\varphi):\varphi\in X}$ is called a \defnotion{theory} of the logic $\bm{C}$ with the set of \defnotion{axioms} $X$.
The set $C(\varnothing)$ is called the \defnotion{set of tautologies} of the logic, and its elements are called \defnotion{tautologies}. A theory $C(X)$ is called \defnotion{consistent} if $C(X)\not=\lang{L}_{\lang{S}}$. A logic $\bm{C}$ is called \defnotion{consistent} if the set $C(\varnothing)$ is consistent.

A standard consequence $C$ allows one to define on the set $\lang{L}_{\lang{S}}$ of all $\lang{S}$\nobreakdash-formulas the \defnotion{logical consequence relation} $\vdash_C~ \subseteq \mathcal{P}_{\mathit{fin}}(\lang{L}_{\lang{S}})\times\lang{L}_{\lang{S}}$ as follows:
$$
\begin{array}{lcl}
\Gamma\vdash_C\varphi
  & \bydef
  & \varphi\in C(\Gamma).
\end{array}  
$$
Conversely, for a logical consequence relation $\vdash$, one can define a consequence $C_\vdash$ by putting
$$
\begin{array}{lcl}
C_\vdash(\Gamma)
  & =
  & \set{\varphi\in \lang{L}_{\lang{S}} : \Gamma\vdash\varphi}.
\end{array}  
$$
Thus, the definitions of a logic as a pair $\langle \lang{S}, C\rangle$ and as a pair $\langle \lang{S}, \vdash\rangle$ are equivalent.

A pair $(\Gamma, \varphi)$, where $\Gamma\cup\{\varphi\}\in\mathcal{P}_{\mathit{fin}}(\lang{L}_\lang{S})$, is called a \defnotion{sequent}. A sequent $(\Gamma, \varphi)$ is called a \defnotion{sequent of the logic} $\langle \lang{S},C \rangle$ if $(\Gamma, \varphi)\in {\vdash_C}$; in this case we write $\Gamma\vdash_C\varphi$.

\defnotion{Inference rules} are understood and denoted in the standard way---as schemata in the metalanguage \cite[7.2, p.~24]{IG:MR:1}. For writing rules we use the language $\lang{L}_{\lang{S}}$, where propositional variables are understood as metavariables for formulas; for this reason we assume that the premises and conclusions of inference rules contain $\lang{L}_{\lang{S}}$\nobreakdash-formulas. Premise-free inference rules are called \defnotion{axiom schemata} \cite[7.6, p.~25]{IG:MR:1}, or simply \defnotion{axioms}\footnote{With this, no confusion with the axioms of the set $C(X)$ arises.}; axioms are identified with the conclusions of the corresponding premise-free rules.

Let $A$ be a set of axioms, $R$ a set of inference rules not containing axioms, and $F$ a set of $\lang{L}_{\lang{S}}$\nobreakdash-formulas. We say that $A$ \defnotion{axiomatizes $F$ over $R$} if $F$ is the smallest set of $\lang{L}_{\lang{S}}$\nobreakdash-formulas containing $A$ and closed under all inference rules from $R$.
We say that $A$ \defnotion{axiomatizes the logic $\langle \lang{S},C \rangle$ over $R$} if for any set $X$ of $\lang{L}_{\lang{S}}$\nobreakdash-formulas the set $C(X)$ is the smallest among the sets containing $X\cup A$ and closed under all inference rules from $R$.

Let $\Gamma$ be a nonempty finite set of $\lang{S}$-formulas and $\alpha$ an $\lang{S}$\nobreakdash-formula. By $\Gamma^\ast_{\hspace{-.16ex}\wedge}$ we denote the set of all possible multiple conjunctions composed of all formulas occurring in $\Gamma$.
Define the following set of implications:
$$
\begin{array}{lcl}
[\Gamma \to \alpha]^\wedge & = & \set{\varphi\to\alpha : \varphi\in \Gamma^\ast_{\hspace{-.16ex}\wedge}}.
\end{array}
$$ 

From now on we consider $\lang{S}$ and $C$ fixed. Thus we fix the language $\lang{L}_\lang{S}$ and also the logic $\bm{C}=\langle\lang{S},C\rangle$. For this reason we simplify the notation and terminology: we write $\lang{L}$ instead of $\lang{L}_\lang{S}$, identify the logic $\bm{C}$ with the consequence $C$, and call $\lang{S}$\nobreakdash-formulas simply formulas. Moreover, we sometimes write a finite set of formulas by listing its elements, omitting the curly braces.

\section{Well-determined logics and the minimal well-deter\-mined logic}
   \setcounter{equation}{0}
   \label{sec:3}

R. W\'{o}jcicki~\cite[Chapter~3]{IG:MR:1} called a logic $C$ \defnotion{well-determined} if it is compatible with the connectives $\wedge$ and $\to$ in the following way:
\begin{itemize}[leftmargin=4em, noitemsep, topsep=5pt, parsep=1pt]
\item[(B1)] $\beta\to\alpha\in C(\varnothing) ~\iff~ \alpha\in C(\beta)$;
\item[(B2)] $C(\alpha\wedge\beta)=C(\alpha,\beta)$.
\end{itemize}
As noted above, in any well-determined logic, to a sequent of the form $(\Gamma, \alpha)$ corresponds the set of formulas $[\Gamma\to\alpha]^\wedge$. This property was taken as the basis for the definition of these logics in~\cite{IG:MR:3}; we use it.
Namely, a \defnotion{well-determined logic} is a logic that satisfies\insMR{, for every non-empty finite $\Gamma$,} the following condition\footnote{The letters WD stand for ``well-determined''.}:
\begin{equation}
\tag{{WD}}
\label{eq:WD}
[\Gamma\to\alpha]^\wedge\subseteq C(\varnothing) ~~\iff~~ \alpha\in C(\Gamma).
\end{equation}

\insMR{
Let us show that \eqref{eq:WD} is equivalent to {\rm{(B1)}} and {\rm{(B2)}}; see also \cite[Theorem~5]{IG:MR:3}. Indeed:
\begin{itemize}
\item
{\rm{(B1)}} is an instance of \eqref{eq:WD} with $\Gamma=\{\beta\}$.
\item
Since $\{\alpha,\beta\}^\ast_{\hspace{-.16ex}\wedge} = \{\alpha\wedge\beta\}^\ast_{\hspace{-.16ex}\wedge}$, we obtain, by \eqref{eq:WD},
$$ 
\begin{array}{lclcl}
\gamma\in C(\alpha,\beta) 
  & \iff 
  & \set{\alpha\wedge\beta\to\gamma}\subseteq C(\varnothing)
  & \iff
  & \gamma\in C(\alpha\wedge\beta),
 \end{array}
$$
and {\rm{(B2)}} follows.
\item
To infer $(\Rightarrow)$ of \eqref{eq:WD} from {\rm{(B1)}} and {\rm{(B2)}}, let $[\Gamma\to\alpha]^\wedge\subseteq C(\varnothing)$. This means that $\beta\to\alpha\in C(\varnothing)$, for every $\beta\in\Gamma^\ast_{\hspace{-.16ex}\wedge}$. By {\rm{(B1)}}, $\alpha\in C(\beta)$. By {\rm{(B2)}}, $C(\Gamma)=C(\beta)$. Thus, $\alpha\in C(\Gamma)$.
\item
To infer $(\Leftarrow)$ of \eqref{eq:WD} from {\rm{(B1)}} and {\rm{(B2)}}, let $\alpha\in C(\Gamma)$. By $(B2)$, $C(\Gamma)=C(\beta)$, for every $\beta\in\Gamma^\ast_{\hspace{-.16ex}\wedge}$, and, by $(B1)$, $\beta\to\alpha\in C(\varnothing)$. Thus, $[\Gamma\to\alpha]^\wedge\subseteq C(\varnothing)$.
\end{itemize}
}

Note that {\rm{(B2)}} implies that
$$
\begin{array}{rcl}
C(\alpha\wedge\alpha) & = & C(\alpha); \\
C(\alpha\wedge\beta)  & = & C(\beta\wedge\alpha); \\
C((\alpha\wedge\beta)\wedge\gamma) & = & C(\alpha\wedge(\beta\wedge\gamma)), \\
\end{array}
$$
i.e., one can say that conjunction possesses the properties inherent to it in classical logic.

With such a definition of a well-determined logic, the operation $\vec L$ defined in~\eqref{eq:7} is naturally defined as follows:
\begin{equation}
\label{eq:8}
\alpha\in\vec L(X)
  ~~\iff~~
  \mbox{there is $\Gamma \in \mathcal{P}_{\mathit{fin}}(X\cup L)$ such that $[\Gamma\to\alpha]^\wedge\subseteq L$.}
\end{equation}

R. W\'{o}jcicki proposed to call a set of formulas $L$ \defnotion{deductive} if there exists a well-determined logic $C$ such that $C(\varnothing)=L$.
Note that for any deductive set $L$, such a logic is uniquely determined and coincides with the logic $\vec L$~\cite[Theorem~9]{IG:MR:3}.

\insMR{The criterion for the deductiveness of a set of formulas presented in \cite[Theorem~11.5]{IG:MR:1} contains an infinite set of inference rules, which hinders its effective application. Therefore, instead of this criterion, we will use an effective criterion formulated in \cite{IG:MR:4}, slightly strengthening it.}
For a set of formulas $L$ we define the following conditions:
\begin{itemize}[leftmargin=4em, noitemsep, topsep=5pt, parsep=1pt]
\item[(C1)] $L$ is \insMR{invariant}; 
\item[(C2)] $\{p\to p\wedge p, p\wedge q\to q\}\subseteq L$;
\item[(C3)] $L$ is closed under the following inference rules:\footnote{We give a comment to the names used for them: 
$(\mathit{TR})$ is \defnotion{transitivity} of implication;
$(\mathit{AD})$ is \defnotion{adjunction};
$(\mathit{MP})$ is \defnotion{modus ponens};
$(\mathit{CM})$ is \defnotion{composition}, aka \defnotion{conjunction composition}, \defnotion{monotonicity of conjunction} or \defnotion{conjunction weakening};
$(\mathit{CV})$ is \defnotion{conversion}, aka \defnotion{conditional conversion} or \defnotion{simplification of antecedent}; 
$(\mathit{EA})$ is \defnotion{extension of antecedent}, aka \defnotion{weakening the antecedent} or \defnotion{left weakening}.}
    $$
    \begin{array}{llcll}
    (\mathit{TR}) & \displaystyle \frac{p\to q, q\to r}{p\to r}; &\mbox{{}$\quad${}}&
    (\mathit{CM}) & \displaystyle \frac{p_1\to q_1, p_2\to q_2}{p_1\wedge p_2\to q_1\wedge q_2};
    \bigskip\\
    (\mathit{AD}) & \displaystyle \frac{p, q}{p\wedge q}; &\mbox{{}$\quad${}}&
    (\mathit{CV}) & \displaystyle \frac{p, p\wedge q\to r}{q\to r};
    \bigskip\\
    (\mathit{MP}) & \displaystyle \frac{p, p\to q}{q}; &\mbox{{}$\quad${}}&
    (\mathit{EA}) & \displaystyle \frac{p\to q}{p\wedge r\to q}.
    \end{array}
    $$
\end{itemize}

\begin{improposition}
\label{prop:dedset:C1-C3}
\textit{A set of formulas $L$ is deductive if and only if $L$ satisfies conditions {\rm{(C1)}}, {\rm{(C2)}}, and~{\rm{(C3)}}.}
\end{improposition}

\begin{proof}
This statement is essentially proved in~\cite[Theorem~3]{IG:MR:4}; \insMR{also, it follows from Lemmas~\ref{2:IG} and~\ref{3:IG}}. Namely, in~\cite{IG:MR:4} condition (C2) is slightly different, requiring additionally that $\{p\to p, p\to p\wedge p, p\wedge q\to q\}\subseteq L$, i.e., it demands the presence of the formula $p\to p$ in $L$. It remains to note that the formula $p\to p$ is derivable from the others and can be omitted,
\insMR{see Lemma~\ref{1:IG}.}
\end{proof}

\delMR{Let us introduce one more inference rule:
$$
\begin{array}{ll}
     (\mathit{PR}) & \displaystyle \frac{q}{p\to q}.
\end{array}
$$

\begin{imlemma} \label{l2}
\textit{Let $L$ be a deductive set. Then $L$ is closed under the rule $(\mathit{PR})$.}
\end{imlemma}

\begin{proof}
Let $\varepsilon q\in L$ for some substitution $\varepsilon$. Recall that $L=\vec L(\varnothing)$. By monotonicity we obtain that $\varepsilon q\in \vec L(\varepsilon p)$. Assuming that $\varepsilon p\to\varepsilon q\not\in L$ leads to a contradiction with~{\rm{(B1)}}. Hence, $\varepsilon p\to\varepsilon q\in L$.
\end{proof}
}

\begin{imlemma}
\label{l1}
\textit{Let $L$ be a deductive set, $\beta,\gamma\in \Gamma^\ast_{\hspace{-.16ex}\wedge}$, and $\beta\to\alpha\in L$. Then $\gamma\to\alpha\in L$.}
\end{imlemma}

\begin{proof}
Note that the formulas
\begin{equation}
\label{eq:lem:l1}
p\wedge q\leftrightarrow q\wedge p, \quad (p\wedge q)\wedge r\leftrightarrow p\wedge (q\wedge r) \quad \mbox{and}  \quad p\wedge p\leftrightarrow p
\end{equation}
belong to $L$ \cite{IG:MR:4}; \insMR{see also Lemma~\ref{1:IG}}. It follows that $\beta\leftrightarrow\gamma\in L$. Applying the rule $(\mathit{TR})$ to the formulas $\beta\to\gamma$ and $\gamma\to\alpha$, we obtain that $\gamma\to\alpha\in L$.
\end{proof}

Proposition~\ref{prop:dedset:C1-C3} allows us to use the following notation: for a set of formulas $\Gamma$, by ${\bigwedge}\Gamma$ we denote some conjunction of all formulas from $\Gamma$. Thanks to Proposition~\ref{prop:dedset:C1-C3}, for a deductive set $L$ the following equivalence holds:
\begin{equation}
\label{eq2:lem:l1:IG}
[\Gamma\to\alpha]^\wedge \subseteq L ~~\iff~~ \mbox{${\bigwedge}\Gamma\to\alpha \in L$.}
\end{equation}
Then from Lemma~\ref{l1} it follows that we can replace the definition \eqref{eq:8} of the operation $\vec L$ by the following:
\begin{equation}
\label{eq:8a}
\alpha\in\vec L(X)
  ~~\iff~~
  \parbox[t]{205pt}{there exists $\Gamma \in \mathcal{P}_{\mathit{fin}}^+(X\cup L)$ such that ${\bigwedge}\Gamma\to\alpha\in L$,}
\end{equation}
\insMR{where ${\bigwedge}\Gamma$ contains no repeated conjuncts.}

\delMR{By ${W}$ we denote the least well-determined logic in the alphabet \mbox{$\langle \Pi, \{\wedge, \to\}, \Upsilon\rangle$}, and from now on we assume that $\lang{S}$ is exactly this alphabet; the logic ${W}$ is called the \defnotion{minimal well-determined logic}. From the criterion for the deductiveness of a set it follows that the logic $W$ can be specified as follows:
\begin{itemize}[noitemsep, topsep=5pt, parsep=1pt]
\item
its set of tautologies, which we denote by $\logic{W}$, is axiomatized by the axioms
$$
\mathit{Ax}1 ~=~ p\to p\wedge p ~~~~\mbox{and}~~~~ \mathit{Ax}2 ~=~ p\wedge q\to q
$$
over the set of rules $\lang{R} = \{(\mathit{TR}), (\mathit{CM}), (\mathit{AD}), (\mathit{CV}), (\mathit{MP}), (\mathit{EA})\}$;
\item every theory of the logic ${W}$ is closed under $(\mathit{AD})$ and under the inference rules from the set
$$
\begin{array}{lcl}
\lang{R}_{\logic{W}}
  & =
  & \displaystyle
    \left\{\frac{\,\Gamma\,}{\alpha} : \mbox{${\bigwedge}\Gamma\to\alpha\in \logic{W}$}\right\}.
\end{array}
$$
\end{itemize}
Note that for the logic ${W}$ the operation $\vec W$ is defined in accordance with \eqref{eq:8}, where one must take the set $\logic{W}$ as~$L$.
}

\insMR{By $\logic{W}$ we denote the least deductive set of formulas in the alphabet \mbox{$\langle \Pi, \{\wedge, \to\}, \Upsilon\rangle$}, and from now on we assume that $\lang{S}$ is exactly this alphabet. 

From the criterion for the deductiveness it follows that $\logic{W}$ is axiomatized by the axioms
$$
\mathit{Ax}1 ~=~ p\to p\wedge p ~~~~\mbox{and}~~~~ \mathit{Ax}2 ~=~ p\wedge q\to q
$$
over the set of rules $\lang{R} = \{(\mathit{TR}), (\mathit{CM}), (\mathit{AD}), (\mathit{CV}), (\mathit{MP}), (\mathit{EA})\}$.

Note that for $\logic{W}$ the operation $\vec{\logic{W}}$ is defined in accordance with \eqref{eq:8}.\footnote{\insMR{The set $\logic{W}$ is the least deductive set when the correspondence between formulas and sequents is given by condition~\eqref{eq:8}.
If conditions other than \eqref{eq:8} are used, the criterion for deductiveness may change.}} 

\insMR{For any deductive set $L$, the logic $\vec L$ is the unique well-determined logic for which $L$ is the set of its tautologies~\cite[Theorem~9]{IG:MR:3}; \insMR{see also Theorem~\ref{4:IG}}. Thus, the operation $\vec{\logic{W}}$ is a well-determined logic. We denote it by $W$ and call it the \defnotion{minimal well-determined logic}.} 
}

\section{Another axiomatization of $\logic{W}$} 
   \setcounter{equation}{0}
   \label{sec:4}

Let $\logic{W}'$ be the set of formulas in the alphabet $\lang{S}=\langle \Pi, \{\wedge, \to\}, \Upsilon\rangle$ derivable in the calculus with the set of axioms $\lang{A}=\{\mathit{Ax}1, \mathit{Ax}2\}$ and the set of inference rules $\lang{R}' = \lang{R}\setminus \{(\mathit{MP})\}$.

To each formula $\varphi$ we associate a maximal (by cardinality) set of formulas $\Gamma$ such that $\varphi\in \Gamma^\ast_{\hspace{-.16ex}\wedge}$; in other words, $\Gamma$ consists of the conjunctive members of the formula $\varphi$ that are not conjunctions of other formulas (in particular, if $\varphi$ is not a conjunction of formulas, then $\Gamma=\{\varphi\}$). The elements of such a set $\Gamma$ are called \defnotion{elementary conjuncts} of the formula $\varphi$. Henceforth, the set of elementary conjuncts of a formula $\varphi$ is denoted by $K(\varphi)$; in particular, $\varphi\in (K(\varphi))^\ast_{\hspace{-.16ex}\wedge}$. It is easy to see that if $L$ is a deductive set, then
$$
\begin{array}{c}
\mbox{$\varphi\leftrightarrow{\bigwedge} K(\varphi) \in L.$}
\end{array}
$$
For a nonempty finite set $\Gamma$ we put
$$
\begin{array}{lcl}
K(\Gamma) & = & \bigcup\limits_{\mathclap{\varphi\in\Gamma}}K(\varphi).
\end{array}
$$
Then it is clear that for a deductive set $L$
\begin{equation}
\label{eq:mr:1}
\mbox{${\bigwedge}\Gamma\leftrightarrow{\bigwedge} K(\Gamma) \in L.$}
\end{equation}

The following assertion is a routine set-theoretic observation.

\begin{imlemma} \label{l3}
\textit{Let $X$, $Y$ and $Z$ be sets and $X\setminus Z\subseteq Y\setminus Z$. Then $X\setminus Y\subseteq Z$ and $X\subseteq Y\cup Z$.}
\end{imlemma}

\begin{proof}
It suffices to note that the formulas $(x\wedge \neg z\to y\wedge \neg z)\to(x\wedge\neg y\to z)$ and $(x\wedge \neg z\to y\wedge \neg z)\to (x\to y\vee z)$, where \linebreak[3] $x$, $y$ and $z$ are propositional variables, are classical tautologies, and then use the completeness of classical logic with respect to set-theoretic semantics.
\end{proof}

The following lemma would be trivial if the set $\lang{R}'$ contained the rule $(\mathit{MP})$; nevertheless, as we shall see, it is not required for the proof.

\begin{imlemma}
\label{l4a}
\textit{Let $\varphi\in \logic{W}'$. Then $K(\varphi)\subseteq \logic{W}'$.}
\end{imlemma}

\begin{proof}
Induction on the length of derivation of the formula $\varphi$ from the axioms in $\lang{A}$ using rules from $\lang{R}'$.

If $\varphi$ is an axiom from $\lang{A}$ or is obtained by one of the rules from $\lang{R}'\setminus \{(\mathit{AD})\}$, then $\varphi$ is an implication of two formulas, hence $K(\varphi) = \{\varphi\}$, and therefore $K(\varphi)\subseteq \logic{W}'$.

If $\varphi$ is obtained by the rule $(\mathit{AD})$, then $\varphi=\varphi_1\wedge\varphi_2$ for some $\varphi_1$ and $\varphi_2$. By the induction hypothesis, $K(\varphi_1)\subseteq \logic{W}'$ and $K(\varphi_2)\subseteq \logic{W}'$. Since $K(\varphi) = K(\varphi_1)\cup K(\varphi_2)$, we obtain that $K(\varphi)\subseteq \logic{W}'$.
\end{proof}

\insMR{Introduce the following inference rule:\footnote{It is known as \defnotion{paradox of implication}, \defnotion{ex falso quodlibet} or \defnotion{principle of explosion}.}
$$
\begin{array}{ll}
     (\mathit{PR}) & \displaystyle \frac{q}{p\to q}.
\end{array}
$$

\begin{imlemma} \label{l2}
\textit{Let $L$ be a deductive set. Then $L$ is closed under\/ $(\mathit{PR})$.}
\end{imlemma}

\begin{proof}
Let $\varepsilon q\in L$ for some substitution $\varepsilon$. Recall that $L=\vec L(\varnothing)$. By monotonicity we obtain that $\varepsilon q\in \vec L(\varepsilon p)$. Assuming that $\varepsilon p\to\varepsilon q\not\in L$ leads to a contradiction with~{\rm{(B1)}}. Hence, $\varepsilon p\to\varepsilon q\in L$.
\end{proof}
}

\begin{imlemma}
\label{l4}
\textit{For any formulas $\alpha$ and $\beta$ the following equivalence holds:
\begin{equation}
\tag{\mbox{$\ast$}}\label{eq:ast}
\begin{array}{lcl}
\alpha\to\beta\in \logic{W}' 
  & \iff 
  & K(\beta)\setminus \logic{W}'\subseteq K(\alpha)\setminus \logic{W}'.
\end{array}
\end{equation}}
\end{imlemma}

\begin{proof}
We prove the implication $(\Rightarrow)$ in~\eqref{eq:ast} by induction on the length of derivation of the formula $\alpha\to\beta$ from the axioms in $\lang{A}$ using rules from $\lang{R}'$. 

{\bf Induction basis.} The set $\lang{A}$ contains two axioms: $\mathit{Ax}1$ and $\mathit{Ax}2$. Consider each of them.
\begin{itemize}[leftmargin=4em, noitemsep, topsep=5pt, parsep=1pt]
\item[$\mathit{Ax}1$:]
In this case $\alpha\to\beta=\alpha\to\alpha\wedge\alpha$. Then $K(\beta)=K(\alpha\wedge\alpha)=K(\alpha)$, hence $K(\beta)\setminus \logic{W}' = K(\alpha)\setminus \logic{W}'$.
\item[$\mathit{Ax}2$:]
In this case $\alpha\to\beta=\gamma\wedge\beta\to\beta$ for some formula $\gamma$. Then $K(\beta)\subseteq K(\gamma\wedge\beta)=K(\alpha)$, hence $K(\beta)\setminus \logic{W}' \subseteq K(\alpha)\setminus \logic{W}'$.
\end{itemize}

{\bf Induction step.}
We prove that the implication $(\Rightarrow)$ in~\eqref{eq:ast} holds for every formula of the form $\alpha\to\beta$ obtained by some inference rule from $\lang{R}'$, under the assumption that for each implication of formulas occurring in the premise of the corresponding rule, the implication $(\Rightarrow)$ in~\eqref{eq:ast} holds.

\begin{itemize}[leftmargin=4em, noitemsep, topsep=5pt, parsep=1pt]
\item[$(\mathit{TR})$:]
Suppose the formula $\alpha\to\beta$ is obtained from formulas $\alpha\to\gamma$ and $\gamma\to \beta$, where $\alpha\to\gamma\in \logic{W}'$ and $\gamma\to \beta\in \logic{W}'$.
By the induction hypothesis, 
$K(\beta)\setminus \logic{W}'\subseteq K(\gamma)\setminus \logic{W}'$ and
$K(\gamma)\setminus \logic{W}'\subseteq K(\alpha)\setminus \logic{W}'$, whence we obtain $K(\beta)\setminus \logic{W}'\subseteq K(\alpha)\setminus \logic{W}'$.
\item[$(\mathit{CM})$:]
Suppose the formula $\alpha\to\beta$ is obtained from formulas $\alpha_1\to\beta_1$ and $\alpha_2\to \beta_2$, where $\alpha_1\to\beta_1\in \logic{W}'$ and $\alpha_2\to \beta_2\in \logic{W}'$; in particular, $\alpha=\alpha_1\wedge\alpha_2$ and $\beta=\beta_1\wedge\beta_2$.
By the induction hypothesis, we have the inclusions
$K(\beta_1)\setminus \logic{W}'\subseteq K(\alpha_1)\setminus \logic{W}'$ and
$K(\beta_2)\setminus \logic{W}'\subseteq K(\alpha_2)\setminus \logic{W}'$, whence we obtain
$$
(K(\beta_1)\setminus \logic{W}')\cup (K(\beta_2)\setminus \logic{W}')~\subseteq~ (K(\alpha_1)\setminus \logic{W}')\cup (K(\alpha_2)\setminus \logic{W}'),
$$
or, equivalently,
$$
(K(\beta_1)\cup K(\beta_2))\setminus \logic{W}'~\subseteq~ (K(\alpha_1)\cup K(\alpha_2))\setminus \logic{W}',
$$
or
$$
K(\beta_1\wedge\beta_2)\setminus \logic{W}'~\subseteq~ K(\alpha_1\wedge\alpha_2)\setminus \logic{W}',
$$
i.e., $K(\beta)\setminus \logic{W}'~\subseteq~ K(\alpha)\setminus \logic{W}'$.

\item[$(\mathit{CV})$:]
Suppose the formula $\alpha\to\beta$ is obtained from formulas $\gamma$ and $\gamma\wedge\alpha\to\beta$. By the induction hypothesis, $K(\beta)\setminus \logic{W}'\subseteq K(\gamma\wedge\alpha)\setminus \logic{W}'$. Moreover, $K(\gamma)\subseteq \logic{W}'$ by Lemma~\ref{l4a}. Then $K(\gamma\wedge\alpha)\setminus \logic{W}' = K(\alpha)\setminus \logic{W}'$, and hence $K(\beta)\setminus \logic{W}'\subseteq K(\alpha)\setminus \logic{W}'$.

\item[$(\mathit{EA})$:]
Suppose the formula $\alpha\to\beta$ is obtained from the formula $\gamma\to\beta$. In this case $\alpha = \gamma\wedge\delta$ for some formula $\delta$. By the induction hypothesis, we have the inclusion $K(\beta)\setminus \logic{W}'\subseteq K(\gamma)\setminus \logic{W}'$, and hence $K(\beta)\setminus \logic{W}'\subseteq K(\gamma\wedge\delta)\setminus \logic{W}' = K(\alpha)\setminus \logic{W}'$.

\item[$(\mathit{AD})$:]
Trivial, because the formula in the conclusion of the rule is not an implication of formulas.
\end{itemize}

We prove the implication $(\Leftarrow)$ in~\eqref{eq:ast}.

Assume $K(\beta)\setminus \logic{W}'\subseteq K(\alpha)\setminus \logic{W}'$.
Put $\Theta = K(\beta)\cap K(\alpha)$, and consider two cases: $\Theta=\varnothing$ and $\Theta\ne\varnothing$.

Suppose $\Theta=\varnothing$.
Then $K(\beta)\setminus \logic{W}'=\varnothing$, because if $\varphi\in K(\beta)\setminus \logic{W}'$, then $\varphi\in K(\alpha)\setminus \logic{W}'$, and hence $\varphi\in \Theta$, which is impossible. Therefore,
$K(\beta)\subseteq \logic{W}$, and taking into account Lemma~\ref{l2}, by $(\mathit{PR})$ we obtain $\beta\in\logic{W}'$. Since $\beta\to\beta\in\logic{W}'$, by $(\mathit{EA})$ we get $\beta\wedge\alpha\to\beta\in\logic{W}'$, and then $\alpha\to\beta\in\logic{W}'$ by $(\mathit{CV})$.

Suppose $\Theta\ne\varnothing$.
Put $\Lambda=K(\beta)\setminus K(\alpha)$ and $\Xi=K(\alpha)\setminus K(\beta)$. Then $\Theta\cup\Lambda=K(\beta)$ and $\Theta\cup\Xi=K(\alpha)$.
Note that $\bigwedge(\Xi\cup\Theta\cup\Lambda)\to\bigwedge(\Theta\cup\Lambda)\in \logic{W}'$ by axiom ${\mathit{Ax}.2}$. By Lemma~\ref{l3}, $\Lambda\subseteq \logic{W}'$. Using $(\mathit{AD})$ we obtain ${\bigwedge}\Lambda\in \logic{W}'$.
Applying $(\mathit{CV})$, we get that
$\bigwedge(\Xi\cup\Theta)\to\bigwedge(\Theta\cup\Lambda)\in \logic{W}'$. It remains to note that
$\bigwedge(\Xi\cup\Theta)\to\bigwedge(\Theta\cup\Lambda)={\bigwedge} K(\alpha)\to {\bigwedge} K(\beta)$.
\end{proof}

\begin{imlemma}
\label{lem:V:consistent}
{Set $\logic{W}'$ contains no atomic formulas.}
\end{imlemma}

\begin{proof}
It suffices to note that $\logic{W}'\subseteq \logic{Cl}$, and $\logic{Cl}$ contains no atomic formulas in the language $\lang{S} = \langle \Pi, \{\wedge, \to\}, \Upsilon\rangle$, since none of them is identically true.
\end{proof}

\begin{imcorollary}
\label{cor:lem4}
\textit{Every formula in $\logic{W}'$ is either an implication of formulas satisfying condition~\eqref{eq:ast} or a conjunction of such implications.}
\end{imcorollary}

\begin{samepage}
Corollary~\ref{cor:lem4} motivates the introduction of the following notion. A set of formulas $U$ in the language $\lang{S} = \langle \Pi, \{\wedge, \to\}, \Upsilon\rangle$ is called \defnotion{well-structured}, or \defnotion{adequate}, if it satisfies the following conditions:
\begin{itemize}[noitemsep, topsep=5pt, parsep=1pt]
\item
if $\alpha\to\beta \in U$, then $K(\beta)\setminus U\subseteq K(\alpha)\setminus U$;
\item
$\alpha,\beta\in U$ if and only if $\alpha\wedge\beta\in U$;
\item
if $\varepsilon\in \lang{E}_{\!\lang{S}}$, then $\varepsilon U\subseteq U$.
\end{itemize}
\end{samepage}
In other words, a well-structured set $U$ is obtained from some set of implications of formulas satisfying the implication $(\Rightarrow)$ in~\eqref{eq:ast} (with $\logic{W}'$ replaced by $U$) by closing it under taking conjunctions of formulas, under extracting conjunctive members of formulas, and under all possible substitutions.

\begin{imlemma} \label{l5}
\textit{Let $U$ be an adequate set of formulas. Then $U$ is closed under\/ $(\mathit{MP})$.}
\end{imlemma}

\begin{proof}
Let $\alpha,\alpha\to\beta\in U$. We show that $\beta\in U$. Taking into account that $\alpha\in U$, we obtain $K(\alpha)\subseteq U$, because $U$ contains only those conjunctions of formulas whose conjunctive members lie in $U$. Since $\alpha\to\beta\in U$, we have the inclusion $K(\beta)\setminus U\subseteq K(\alpha)\setminus U$, and by Lemma~\ref{l3} we get $K(\beta)\subseteq K(\alpha)\cup U$. But then $K(\beta)\subseteq K(\alpha)\cup U = U$. Finally, from $K(\beta)\subseteq U$ it follows that $\beta\in U$, since $\beta\in (K(\beta))^\ast_{\hspace{-.16ex}\wedge}$.
\end{proof}

From Lemmas~\ref{l4} and~\ref{l5} we obtain the following statement, which means that $(\mathit{MP})$ is eliminable in the definition of~$\logic{W}$:

\begin{imcorollary}
\label{c1}
\textit{$\logic{W}'=\logic{W}$.}
\end{imcorollary}

\begin{proof}
By the definitions of the sets $\logic{W}'$ and $\logic{W}$, the inclusion $\logic{W}'\subseteq \logic{W}$ holds. Note that by Lemma~\ref{l4}, the set $\logic{W}'$ is adequate. Then from Lemma~\ref{l5} it follows that $\logic{W}'=\logic{W}$.
\end{proof}

Another corollary of Lemmas~\ref{l4} and~\ref{l5} (or of Corollary~\ref{c1} obtained from them) is the following theorem.

\begin{imtheorem}
\label{t1}
\textit{Logic $W$ is consistent.}
\end{imtheorem}

\begin{proof}
This follows from Corollary~\ref{c1} together with Lemma~\ref{lem:V:consistent}.
\end{proof}

\section{Theories of $\vec{\logic{W}}$ and axiomatization of $\vec{\logic{W}}$}
   \setcounter{equation}{0}
   \label{sec:before:5}

\insMR{We have defined the set of tautologies of the logic $W$, and our next goal is to describe the structure of the theories of this logic. We specify a set of inference rules with respect to which these theories are closed; this allows us to present the logic $W$ as a calculus.}

\insMR{
\begin{imtheorem}
\label{t2}
\textit{Let\/ $\Gamma\in \mathcal{P}_{\mathit{fin}}(\lang{L})$ and\/ $\Gamma\cap \logic{W}=\varnothing$. Then
$$
\begin{array}{lcl}
W(\Gamma)
  & =
  & \bigcup \{\Delta^\ast_{\hspace{-.16ex}\wedge} : \Delta \in \mathcal{P}_{\mathit{fin}}^+(\logic{W} \cup K(\Gamma))\}.
\end{array}
$$}
\end{imtheorem}

\begin{proof}
We prove that the specified sets are included in one another, whence their equality follows.

Let $\varphi\in {W}(\Gamma)$. Then, by the definition of the operation $W$, there exists $\Theta\in\mathcal{P}_{\mathit{fin}}^+(\logic{W}\cup\Gamma)$ such that ${\bigwedge}\Theta\to\varphi\in \logic{W}$. By Lemma~\ref{l4} and Corollary~\ref{c1}, $K(\varphi)\setminus \logic{W}\subseteq K(\Theta)\setminus \logic{W}$, and hence $K(\varphi)\cup \logic{W}\subseteq K(\Theta)\cup \logic{W}$. Consequently,
$$
\begin{array}{lclclcl}
K(\varphi)
  & \subseteq
  & \logic{W} \cup K(\varphi)
  & \subseteq
  & \logic{W} \cup K(\Theta)
  & \subseteq
  & \logic{W} \cup K(\Gamma).
\end{array}
$$
Since $\varphi\in (K(\varphi))^\ast_{\hspace{-.16ex}\wedge}$, there exists $\Delta\in\mathcal{P}_{\mathit{fin}}^+((K(\varphi))^\ast_{\hspace{-.16ex}\wedge})$ such that $\varphi\in \Delta^\ast_{\hspace{-.16ex}\wedge}$. Because we have the inclusion $K(\varphi)\subseteq\logic{W} \cup K(\Gamma)$, it follows that $\Delta \in \mathcal{P}_{\mathit{fin}}^+(\logic{W} \cup K(\Gamma))$, and hence $\varphi\in \bigcup \{\Delta^\ast_{\hspace{-.16ex}\wedge} : \Delta \in \mathcal{P}_{\mathit{fin}}^+(\logic{W} \cup K(\Gamma))\}$.

Let $\varphi\in\Delta^\ast_{\hspace{-.16ex}\wedge}$ for some $\Delta \in \mathcal{P}_{\mathit{fin}}^+(\logic{W} \cup K(\Gamma))$. Then $\varphi$ is a multiple conjunction of all formulas from $\Delta$, and hence ${\bigwedge}\Delta\to\varphi\in\logic{W}$, and by~\eqref{eq:8a}, $\varphi\in {W}(K(\Gamma))$. Using~\eqref{eq:mr:1} and $(\mathit{TR})$, we obtain ${W}(K(\Gamma))={W}(\Gamma)$, and consequently $\varphi\in {W}(\Gamma)$.
\end{proof}
}

\insMR{
We introduce two inference rules allowing to remove conjuncts from a conjunction of formulas:\footnote{These are \defnotion{left simplification} and \defnotion{right simplification}, aka \defnotion{left conjunction elimination} and \defnotion{right conjunction elimination}.}
$$
\begin{array}{llcll}
(\mathit{SL}) & \displaystyle \frac{p\wedge q}{q} &\mbox{{}$\quad$and$\quad${}}&
(\mathit{SR}) & \displaystyle \frac{p\wedge q}{p}.
\end{array}
$$
Set $\lang{R}'' = \set{(\mathit{AD}), (\mathit{SL}), (\mathit{SR})}$.

\begin{imtheorem}
\label{t2:IG}
\textit{A set $T$ is a theory of $W$ if and only if $\logic{W}\subseteq T$ and $T$ is closed under the rules from $\lang{R}''$.}
\end{imtheorem}

\begin{proof}
Let $T$ be a theory of $W$, i.e., $T=W(\Gamma)=\vec{\logic{W}}(\Gamma)$ for some~$\Gamma$. By~\eqref{eq:8a}, $\logic{W}\subseteq\vec{\logic{W}}(\Gamma)$, and hence $\logic{W}\subseteq T$.
Let $\Gamma'=\Gamma\setminus\logic{W}$. Then $\Gamma'\cap \logic{W}=\varnothing$, therefore, by Theorem~\ref{t2}, the theory $W(\Gamma')$ is closed under the rules from $\lang{R}''$. It remains to note that $W(\Gamma')=W(\Gamma)=T$.

Suppose $\logic{W}\subseteq T$ and $T$ is closed under the rules from $\lang{R}''$. We show that $T=W(T)$. The inclusion $T\subseteq W(T)$ is obvious, and we need to prove $W(T)\subseteq T$.
Let $\alpha\in W(T)$. If $\alpha\in\logic{W}$, then $\alpha\in T$ since $\logic{W}\subseteq T$. Suppose $\alpha\not\in\logic{W}$. Since $\alpha\in W(T)$, there exists a finite set $\Gamma\subseteq T$ such that ${\bigwedge}\Gamma\to\alpha\in \logic{W}$. Then, by Lemma~\ref{l4} together with Corollary~\ref{c1}, $K(\alpha)\setminus \logic{W} \subseteq K(\Gamma)\setminus \logic{W}$, and hence, by Lemma~\ref{l3}, $K(\alpha) \subseteq K(\Gamma)\cup \logic{W}$. Since $\Gamma\subseteq T$ and $T$ is closed under $(\mathit{SL})$ and $(\mathit{SR})$, we obtain $K(\Gamma)\subseteq T$, and since $\logic{W}\subseteq T$, we have $K(\Gamma)\cup \logic{W}\subseteq T$, and consequently $K(\alpha)\subseteq T$. Applying $(\mathit{AD})$ to $K(\alpha)$ yields $\alpha\in T$.
Thus $T=W(T)$, and $T$ is a theory of~$W$.
\end{proof}

From Theorems~\ref{t2} and~\ref{t2:IG} it directly follows 
\begin{imtheorem}
\label{t3:IG}
\textit{Logic $W$ is defined by the following calculus:
\begin{itemize}[noitemsep, topsep=5pt, parsep=1pt]
\item
the set of axioms\/ $\logic{W}$ of $W$ is the set of all formulas derivable from the axioms $\lang{A}$ over the set of rules~$\lang{R}'$;
\item 
the theories of $W$ are the sets of formulas that contain $\logic{W}$ and are closed under the rules from~$\lang{R}''$.
\end{itemize}}
\end{imtheorem}
}

\section{Frege relation for $\vec{\logic{W}}$}
   \setcounter{equation}{0}
   \label{sec:5}

For a logic $C$ we define the relation $\dashv\vdash_C$: for any formulas $\alpha$ and $\beta$ of the language of $C$ we set
$$
\begin{array}{lcl}
\alpha\dashv\vdash_C\beta & \bydef & \mbox{$\alpha\vdash_C \beta$ and $\beta\vdash_C \alpha$.}
\end{array}
$$
The relation $\dashv\vdash_C$ is called the \defnotion{interderivability relation}\footnote{This relation is sometimes called the \defnotion{Frege relation}, see e.g.~\cite[Section~2.1]{IG:MR:FJP:2003}.} of formulas in $C$, and formulas in this relation are called \defnotion{interderivable} in $C$. 

It is easy to see that the interderivability relation of formulas in a logic $C$ is an equivalence, and hence partitions the set of all formulas of the language into equivalence classes. For a formula $\alpha$, the equivalence class containing $\alpha$ is denoted by $[\alpha]_C$, i.e.,
$$
\begin{array}{lcl}
[\alpha]_C & = & \set{\beta\in\lang{L} : \alpha\dashv\vdash_C\beta}, 
\end{array}
$$
where $\lang{L}$ is the language of $C$.
Let $\lang{L}/C$ be the set of all such equivalence classes, i.e.,
$$
\begin{array}{lcl}
\lang{L}/C & = & \set{[\alpha]_C : \alpha\in \lang{L}}. 
\end{array}
$$

In what follows we take the logic $W$ as $C$; in particular, this means that the language $\lang{L}$ contains only $\wedge$ and $\to$ as connectives. We simplify the notation: we write $\vdash$ instead of $\vdash_W$, and also $\dashv\vdash$ instead of $\dashv\vdash_W$ and $[\alpha]$ instead of $[\alpha]_W$.

The following technical lemma is almost obvious.

\begin{imlemma}
\label{lem:bwedge}
\textit{Let $\alpha\vdash \alpha'$ and $\beta\vdash \beta'$. Then\/ $\alpha\wedge\beta\vdash \alpha'\wedge\beta'$.}
\end{imlemma}

\begin{proof}
Since $\alpha,\beta\vdash\alpha'$ and $\alpha,\beta\vdash\beta'$, by $(\mathit{AD})$ we obtain $\alpha,\beta\vdash\alpha'\wedge\beta'$, and then $\alpha\wedge\beta\vdash\alpha'\wedge\beta'$ by (B2) with $C=W$.
\end{proof}

If $\alpha\dashv\vdash\alpha'$ and $\beta\dashv\vdash\beta'$, then by Lemma~\ref{lem:bwedge}, $[\alpha\wedge\beta]=[\alpha'\wedge\beta']$. This allows us to introduce the following operation $\bwedge$ on $\lang{L}/W$:
\begin{equation}
\label{eq:bwedge}
\begin{array}{lcl}
[\alpha]\bwedge[\beta]& = & [\alpha\wedge\beta].
\end{array}
\end{equation}

Since the formulas~\eqref{eq:lem:l1} belong to the set $\logic{W}$, the operation $\bwedge$ is commutative, associative, and idempotent on $\lang{L}/W$, and hence the structure $\otuple{\lang{L}/W, \bwedge}$ is a semilattice.

On the semilattice $\otuple{\lang{L}/W, \bwedge}$ we introduce a partial order in the usual way:
$$
\begin{array}{lcl}
[\alpha]\leqslant[\beta] & \bydef & [\alpha]\bwedge[\beta]=[\alpha].
\end{array}
$$
It is not hard to see that this semilattice has a greatest element, which is ${W}(\varnothing)$, i.e., $\logic{W}$.

\begin{improposition}
\label{prop:leq}
\label{t4}
\textit{For any formulas $\alpha$ and $\beta$ the following equivalence holds:
\begin{equation}
\tag{\mbox{${\ast}{\ast}$}}\label{eq:2ast}
\begin{array}{lcl}
\alpha\to\beta\in \logic{W} & \iff & [\alpha]\leqslant[\beta].
\end{array}
\end{equation}}
\end{improposition}

\begin{proof}
We prove the implication $(\Rightarrow)$ in \eqref{eq:2ast} by induction on the length of derivation of the formula $\alpha\to\beta$ from the axioms in $\lang{A}$ using rules from $\lang{R}'$.

{\bf Induction basis.} The set $\lang{A}$ contains two axioms: $\mathit{Ax}1$ and $\mathit{Ax}2$. Consider each of them.
\begin{itemize}[leftmargin=4em, noitemsep, topsep=5pt, parsep=1pt]
\item[$\mathit{Ax}1$:]
In this case $\alpha\to\beta=\alpha\to\alpha\wedge\alpha$. Since $[\alpha]=[\alpha\wedge\alpha]=[\beta]$, we obtain $[\alpha]\leqslant[\beta]$.
\item[$\mathit{Ax}2$:]
In this case $\alpha\to\beta=\gamma\wedge\beta\to\beta$ for some formula $\gamma$. Then $[\alpha]\bwedge[\beta]=[\alpha\wedge\beta]=[\gamma\wedge\beta\wedge\beta]=[\gamma\wedge\beta]=[\alpha]$, and hence $[\alpha]\leqslant[\beta]$.
\end{itemize}

{\bf Induction step.}
We prove that the implication $(\Rightarrow)$ in~\eqref{eq:2ast} holds for every formula $\alpha\to\beta$ obtained by some inference rule from $\lang{R}'$, under the assumption that for each implication of formulas occurring in the premise of the corresponding rule, the implication $(\Rightarrow)$ in~\eqref{eq:2ast} holds.

\begin{itemize}[leftmargin=4em, noitemsep, topsep=5pt, parsep=1pt]

\item[$(\mathit{TR})$:]
Suppose the formula $\alpha\to\beta$ is obtained from formulas $\alpha\to\gamma$ and $\gamma\to \beta$, where $\alpha\to\gamma\in \logic{W}$ and $\gamma\to \beta\in \logic{W}$.
By the induction hypothesis, 
$[\alpha]\leqslant[\gamma]$ and $[\gamma]\leqslant[\beta]$. Then by transitivity of the relation $\leqslant$ we obtain $[\alpha]\leqslant[\beta]$.

\item[$(\mathit{CM})$:]
Suppose the formula $\alpha\to\beta$ is obtained from formulas $\alpha_1\to\beta_1$ and $\alpha_2\to \beta_2$, where $\alpha_1\to\beta_1\in \logic{W}$ and $\alpha_2\to \beta_2\in \logic{W}$; in particular, $\alpha=\alpha_1\wedge\alpha_2$ and $\beta=\beta_1\wedge\beta_2$.
By the induction hypothesis, $[\alpha_1]\leqslant[\beta_1]$ and $[\alpha_2]\leqslant[\beta_2]$. Then
$$
[\alpha\wedge\beta] = 
[\alpha_1\wedge\beta_1]\bwedge[\alpha_2\wedge\beta_2] = 
[\alpha_1]\bwedge[\alpha_2] = 
[\alpha], 
$$
and hence $[\alpha]\leqslant[\beta]$.

\item[$(\mathit{CV})$:]
Suppose the formula $\alpha\to\beta$ is obtained from formulas $\gamma$ and $\gamma\wedge\alpha\to\beta$. 
By the induction hypothesis, $[\gamma\wedge\alpha]\leqslant[\beta]$. But $[\gamma]=\logic{W}$, therefore $[\gamma\wedge\alpha]=[\gamma]\bwedge[\alpha]=[\alpha]$
and hence $[\alpha]\leqslant[\beta]$.

\item[$(\mathit{EA})$:]
Suppose the formula $\alpha\to\beta$ is obtained from the formula $\gamma\to\beta$. In this case $\alpha = \gamma\wedge\delta$ for some formula $\delta$. By the induction hypothesis, 
$[\gamma]\leqslant[\beta]$. But $[\gamma\wedge\delta]\bwedge[\gamma]=[\gamma\wedge\delta]$, hence $[\gamma\wedge\delta]\leqslant[\gamma]$, and by transitivity of $\leqslant$ we obtain $[\alpha]\leqslant[\beta]$.

\item[$(\mathit{AD})$:]
Trivial, because the formula in the conclusion of the rule is not an implication of formulas.
\end{itemize}

We prove the implication $(\Leftarrow)$ in~\eqref{eq:2ast}. Suppose $[\alpha]\leqslant[\beta]$. Then $[\alpha]\bwedge[\beta]=[\alpha]$. Since $[\alpha]\bwedge[\beta]=[\alpha\wedge\beta]$, we obtain $[\alpha]=[\alpha\wedge\beta]$, and consequently $\alpha\dashv\vdash\alpha\wedge\beta$. Taking into account that the operation ${W}$ satisfies condition (B1), we conclude $\alpha\to\alpha\wedge\beta\in\logic{W}$. It remains to note that $\alpha\wedge\beta\to\beta\in\logic{W}$, and then by $(\mathit{TR})$ we obtain $\alpha\to\beta\in\logic{W}$.
\end{proof}

\begin{imcorollary}
\label{c3}
\textit{For any formulas $\alpha$ and $\beta$ the following equivalence holds:
$$
\begin{array}{lcl}
[\alpha]\leqslant[\beta] & \iff & [\alpha\to\beta]=\logic{W}.
\end{array}
$$}
\end{imcorollary}

\begin{proof}
It suffices to note that $[\alpha\to\beta]=\logic{W}$ if and only if $\alpha\to\beta\in\logic{W}$.
\end{proof}

Thanks to Lemma~\ref{lem:bwedge}, we were able to define a binary operation on the set $\lang{L}/W$, assigning to the classes $[\alpha]$ and $[\beta]$ the class $[\alpha\wedge\beta]$, see~\eqref{eq:bwedge}. We show that for implication the situation is different: such an operation does not exist.

\begin{improposition}
\label{prop:diff}
\textit{Let $\alpha\to\beta\not\in\logic{W}$, $\alpha'\in[\alpha]$, $\beta'\in[\beta]$, and $\alpha\to\beta\ne\alpha'\to\beta'$. Then $[\alpha\to\beta]\ne[\alpha'\to\beta']$.}
\end{improposition}

\begin{proof}
It suffices to note that
$$
\begin{array}{lclclcl}
K(\alpha\to\beta)
  & =
  & \set{\alpha\to\beta}
  & \ne
  & \set{\alpha'\to\beta'}
  & =
  & K(\alpha'\to\beta'),
\end{array}
$$
and then $[\alpha\to\beta]\ne[\alpha'\to\beta']$ by Lemma~\ref{l4} taking into account (B1).
\end{proof}

Thus, for example, if $p$ and $q$ are distinct propositional variables, then $[p]=[p\wedge p]$, but $[p\to q]\ne[p\wedge p\to q]$. Moreover, Proposition~\ref{prop:diff} implies that if $\alpha\to\beta\not\in\logic{W}$, then the set
$$
\set{[\alpha'\to\beta'] : \text{$\alpha'\in[\alpha]$ and $\beta'\in[\beta]$}}
$$
is infinite: it suffices to take as $\alpha'$ various multiple conjunctions composed of $\alpha$, or as $\beta'$ various multiple conjunctions composed of $\beta$.

From Proposition~\ref{prop:diff} it also follows that there exists a set of formulas $\Gamma$ such that ${W}(\Gamma)$ is not closed under the operation of replacement of equivalent subformulas. Indeed, taking again two variables $p$ and $q$, we have $p\lra p\wedge p \in \logic{W}$, but $p\wedge p\to q\not\in {W}(p\to q)$ and $p\to q\not\in {W}(p\wedge p\to q)$.

However, by Proposition~\ref{prop:leq}, we can define a partial function $\brhu$ on $\lang{L}/W$ by setting for any formulas $\alpha$ and $\beta$
$$
\begin{array}{lcl}
[\alpha]\brhu[\beta]
  & =
  & \left\{
    \begin{array}{rl}
    \logic{W} & \text{if $[\alpha]\leqslant[\beta]$;}\\
    \text{undefined} & \text{if $[\alpha]\not\leqslant[\beta]$.}
    \end{array}
    \right.
\end{array}
$$
As we will see below, this approach is fully justified; moreover, it allows constructing a semantics for the logic~$W$.

\section{Semantics for $\vec{\logic{W}}$}
   \setcounter{equation}{0}
   \label{sec:6}

A tuple $\bm{M}=\otuple{M,\wedge,\top}$ is called a \defnotion{$W$-algebra}, or a \defnotion{$W$-matrix}, if
\begin{itemize}[noitemsep, topsep=5pt, parsep=1pt]
\item $\otuple{M,\wedge}$ is a semilattice;
\item $\top$ is the greatest element in the partially ordered set $\otuple{M,\leqslant}$, where $a\leqslant b$ means $a\wedge b = a$. 
\end{itemize}

A function $v\colon\Pi\to M$ is called a \defnotion{valuation} on the matrix $\bm{M}$. Suppose we have a function $\mu\colon\lang{L}\to M$. We extend the valuation $v$ to a function ${v}_\mu\colon\lang{L}\to M$ as follows:
$$
\begin{array}{lcl}
{v}_\mu(p) & = & v(p),\quad \text{if $p\in\Pi$;}
  \medskip\\
{v}_\mu(\varphi\wedge\psi) 
  & = 
  & {v}_\mu(\varphi)\wedge{v}_\mu(\psi);
    \medskip\\
{v}_\mu(\varphi\to\psi) 
  & = 
  & \left\{
    \begin{array}{rl}
    \top & \text{if ${v}_\mu(\varphi)\leqslant{v}_\mu(\psi)$;}\\
    \mu(\varphi\to\psi) & \text{if ${v}_\mu(\varphi)\not\leqslant{v}_\mu(\psi)$.}
    \end{array}
    \right.
    \\
\end{array}
$$
The function ${v}_\mu$ is called a \defnotion{$\mu$-extension} of the valuation $v$ in the $W$-matrix $\bm{M}$.

We say that a formula $\varphi$ is \defnotion{true on the $W$-matrix $\bm{M}$ under the $\mu$-extension $v$} if ${v}_\mu(\varphi)=\top$; in this case we write $\bm{M}\models^v_\mu\varphi$. 

We say that a formula $\varphi$ is \defnotion{true on the $W$-matrix $\bm{M}$ under the valuation $v$} if $\bm{M}\models^v_\mu\varphi$ for every $\mu$-extension $v$; in this case we write $\bm{M}\models^v\varphi$. 

We say that a formula $\varphi$ is \defnotion{true on the $W$-matrix $\bm{M}$} if $\bm{M}\models^v\varphi$ for every valuation $v$; in this case we write $\bm{M}\models\varphi$.

Let $\mathcal{C}$ be a class of $W$-matrices.
We say that a formula $\varphi$ is \defnotion{true in the class $\mathcal{C}$} if $\varphi$ is true on every $W$-matrix from $\mathcal{C}$; in this case we write $\mathcal{C}\models\varphi$.

Let $\mathcal{W}$ be the class of all $W$-matrices.

\begin{imlemma}
\label{lem:for:t5}
\textit{Let ${v}_\mu$ be a \defnotion{$\mu$-extension} of some valuation $v$ in a $W$-matrix $\bm{M}=\otuple{M,\wedge,\top}$. Then, for every formula $\varphi\in\logic{W}$,
\begin{enumerate}[leftmargin=4em, noitemsep, topsep=5pt, parsep=1pt]
\renewcommand{\labelenumi}{$(\arabic{enumi})$}
\item \label{lem:for:t5:item:1}
$v_\mu(\varphi)=\top$\textup{;}
\item \label{lem:for:t5:item:2}
if $\varphi=\alpha\to\beta$, then $v_\mu(\alpha)\leqslant v_\mu(\beta)$.
\end{enumerate}}
\end{imlemma}

\begin{proof}
We prove \eqref{lem:for:t5:item:1} and \eqref{lem:for:t5:item:2} simultaneously by induction on the derivation of~$\varphi$ from the axioms in~$\lang{A}$ using the rules from the set~$\lang{R}'$.

{\bf Induction basis.} The set $\lang{A}$ contains the axioms $\mathit{Ax}1$ and $\mathit{Ax}2$; we consider each of them.
\begin{itemize}[leftmargin=4em, noitemsep, topsep=5pt, parsep=1pt]
\item[$\mathit{Ax}1$:]
Let $\varphi=\alpha\to\alpha\wedge\alpha$.
Since $v_\mu(\alpha)=v_\mu(\alpha)\wedge v_\mu(\alpha)$, we obtain that both \eqref{lem:for:t5:item:1} and \eqref{lem:for:t5:item:2} hold for $\varphi$.
\item[$\mathit{Ax}2$:]
Let $\varphi=\alpha\wedge\beta\to\beta$.
Clearly, $v_\mu(\alpha\wedge\beta)\wedge v_\mu(\beta)=v_\mu(\alpha\wedge\beta)$, therefore $v_\mu(\alpha\wedge\beta)\leqslant v_\mu(\beta)$, and hence \eqref{lem:for:t5:item:1} and \eqref{lem:for:t5:item:2} hold for $\varphi$.
\end{itemize}

{\bf Induction step.}
We show that every inference rule from $\lang{R}'$ preserves \eqref{lem:for:t5:item:1} and \eqref{lem:for:t5:item:2}.

\begin{itemize}[leftmargin=4em, noitemsep, topsep=5pt, parsep=1pt]

\item[$(\mathit{TR})$:]
Suppose $\varphi$ is obtained from $\alpha\to\gamma$ and $\gamma\to \beta$; in particular, $\varphi=\alpha\to\beta$.
By the induction hypothesis,
$v_\mu(\alpha)\leqslant v_\mu(\gamma)$ and $v_\mu(\gamma)\leqslant v_\mu(\beta)$, hence $v_\mu(\alpha)\leqslant v_\mu(\beta)$, which gives \eqref{lem:for:t5:item:2}; by the definition of $v_\mu$ we obtain~\eqref{lem:for:t5:item:1}.

\item[$(\mathit{CM})$:]
Suppose $\varphi$ is obtained from $\alpha_1\to\beta_1$ and $\alpha_2\to \beta_2$; in particular, $\varphi=\alpha\to\beta$ with $\alpha=\alpha_1\wedge\alpha_2$ and $\beta=\beta_1\wedge\beta_2$.
By the induction hypothesis, $v_\mu(\alpha_1)\leqslant v_\mu(\beta_1)$ and $v_\mu(\alpha_2)\leqslant v_\mu(\beta_2)$. Then
$$
v_\mu(\alpha\wedge\beta) = 
v_\mu(\alpha_1\wedge\beta_1)\wedge v_\mu(\alpha_2\wedge\beta_2) = 
v_\mu(\alpha_1)\wedge v_\mu(\alpha_2) = 
v_\mu(\alpha), 
$$
and hence $v_\mu(\alpha)\leqslant v_\mu(\beta)$, which yields \eqref{lem:for:t5:item:2}, whence \eqref{lem:for:t5:item:1} follows.

\item[$(\mathit{CV})$:]
Suppose $\varphi$ is obtained from $\gamma$ and $\gamma\wedge\alpha\to\beta$; in particular, $\varphi=\alpha\to\beta$.
By the induction hypothesis, $v_\mu(\gamma\wedge\alpha)\leqslant v_\mu(\beta)$ and $v_\mu(\gamma)=\top$, therefore $v_\mu(\gamma\wedge\alpha)=v_\mu(\gamma)\wedge v_\mu(\alpha)=v_\mu(\alpha)$,
and hence $v_\mu(\alpha)\leqslant v_\mu(\beta)$, which gives \eqref{lem:for:t5:item:2}, and then \eqref{lem:for:t5:item:1}.

\item[$(\mathit{EA})$:]
Suppose $\varphi$ is obtained from $\gamma\to\beta$; in particular, $\varphi=\alpha\to\beta$, where $\alpha = \gamma\wedge\delta$ for some~$\delta$. By the induction hypothesis, 
$v_\mu(\gamma)\leqslant v_\mu(\beta)$. But $v_\mu(\gamma\wedge\delta)\wedge v_\mu(\gamma)= v_\mu(\gamma\wedge\delta)$, hence $v_\mu(\gamma\wedge\delta)\leqslant v_\mu(\gamma)$, and then $v_\mu(\alpha)\leqslant v_\mu(\beta)$, which yields \eqref{lem:for:t5:item:2}, and consequently \eqref{lem:for:t5:item:1}.

\item[$(\mathit{AD})$:]
Suppose $\varphi=\gamma\wedge\psi$ and is obtained from $\gamma$ and $\psi$. Then \eqref{lem:for:t5:item:2} holds trivially because $\varphi$ is not an implication of formulas. By the induction hypothesis, $v_\mu(\gamma)=\top$ and $v_\mu(\psi)=\top$, and hence $v_\mu(\varphi)=v_\mu(\gamma)\wedge v_\mu(\psi)=\top$, which gives \eqref{lem:for:t5:item:1}.
\end{itemize}

Thus, \eqref{lem:for:t5:item:1} and \eqref{lem:for:t5:item:2} hold for every formula $\varphi\in\logic{W}$.
\end{proof}

\begin{imtheorem}[Soundness]     
\label{t5}
\textit{If $\varphi\in\logic{W}$, then $\mathcal{W}\models\varphi$.}
\end{imtheorem}

\begin{proof}
Follows from Lemma~\ref{lem:for:t5}\,\eqref{lem:for:t5:item:1}.
\end{proof}

\begin{imtheorem}[Completeness]     
\label{t6a}
\textit{If $\mathcal{W}\models\varphi$, then $\varphi\in\logic{W}$.}
\end{imtheorem}

\begin{proof}
Assume $\varphi\not\in\logic{W}$ and show that $\varphi$ is refuted on the $W$-algebra $\bm{M}_W = \otuple{\lang{L}/W,\bwedge,\logic{W}}$. 

Define a valuation $w$ on the $W$-algebra $\bm{M}_W$ and its $\mu$-extension $w_\mu$ by setting
$$
\begin{array}{lcll}
w(p) 
  & = 
  & [\hfill{p}\hfill], 
  & \quad\text{where $p\in\Pi$}; 
  \smallskip\\
\mu(\psi) 
  & = 
  & [\psi], 
  & \quad\text{where $\psi\in\lang{L}$.}
\end{array}
$$
Then for every formula $\psi\in\lang{L}$
$$
\begin{array}{lcl}
w_\mu(\psi) & = & [\psi]. 
\end{array}
$$
We prove this equality by induction on $\psi$. If $\psi=p$ for some variable $p\in\Pi$, then 
$$
\begin{array}{lclclclcl}
w_\mu(\psi) & = & w_\mu(p) & = & w(p) & = & [p] & = & [\psi]. 
\end{array}
$$
If $\psi=\psi'\wedge\psi''$, then by~\eqref{eq:bwedge} 
$$
\begin{array}{lclclclcl}
w_\mu(\psi) & = & w_\mu(\psi')\bwedge w_\mu(\psi'') & = & [\psi']\bwedge[\psi''] & = & [\psi'\wedge\psi''] & = & [\psi].
\end{array}
$$
Let $\psi=\psi'\wedge\psi''$. If $[\psi']\leqslant[\psi'']$, then $[\psi]=\logic{W}$ by Corollary~\ref{c3}, and hence $w_\mu=[\psi]$. If $[\psi']\not\leqslant[\psi'']$, then $w_\mu=[\psi]$ by the definition of the $\mu$-extension $w$.

Since $\varphi\not\in\logic{W}$, we have $w_\mu(\psi)=[\psi]\ne\logic{W}$. Then $\bm{M}_W\not\models^w_\mu\varphi$, and consequently $\mathcal{W}\not\models\varphi$.
\end{proof}

\section{Decidability of $\logic{W}$}
   \setcounter{equation}{0}
   \label{sec:7}

In this section we show that the set of tautologies of the minimal well-determined logic is decidable and present a corresponding algorithm whose running time is bounded by a polynomial in the length of the tested formula.
The algorithm is based on the following observations:
\begin{itemize}[noitemsep, topsep=5pt, parsep=1pt]
\item
if $\varphi$ is a variable, then $\varphi\not\in\logic{W}$ by Lemma~\ref{lem:V:consistent};
\item
if $\varphi=\alpha\wedge\beta$, then $\varphi\in\logic{W}$ if and only if $\alpha\in\logic{W}$ and $\beta\in\logic{W}$, since $\logic{W}$ is a well-structured set;
\item
if $\varphi=\alpha\to\beta$, then, by Lemma~\ref{l4}, $\varphi\in\logic{W}$ if and only if $K(\beta)\setminus\logic{W}$ is contained in $K(\alpha)\setminus\logic{W}$.
\end{itemize}
These observations allow us to construct a recursive procedure for checking whether a formula $\varphi$ belongs to~$\logic{W}$. Since it will be convenient to rely on more formal constructions in what follows, we proceed to their description.

We describe a recursive procedure $\mathtt{IsWTautology(\mbox{$x$})}$, which takes as input $x$ an arbitrary formula $\varphi$ and outputs $\mathtt{true}$ (if $\varphi\in\logic{W}$) or $\mathtt{false}$ (if $\varphi\notin\logic{W}$).
The procedure works as follows:
\begin{itemize}[noitemsep, topsep=5pt, parsep=1pt]
\item
if $\varphi\in\Pi$, then $\mathtt{IsWTautology(\mbox{$\varphi$})} = \mathtt{false}$;

\item
if $\varphi=\alpha\wedge\beta$, then $\mathtt{IsWTautology(\varphi)}$ is taken to be the conjunction of the values
$\mathtt{IsWTautology(\alpha)}$ and $\mathtt{IsWTautology(\beta)}$;

\item if $\varphi=\alpha\to\beta$, then for each $\gamma\in K(\beta)$ we check whether $\mathtt{IsWTautology(\gamma)} = \mathtt{true}$ or $\gamma\in K(\alpha)$; if this condition holds, then $\mathtt{IsWTautology(\varphi)} = \mathtt{true}$, otherwise $\mathtt{IsWTautology(\varphi)} = \mathtt{false}$.
\end{itemize}
A formal description of the procedure \texttt{IsWTautology}($\varphi$) is presented in pseudocode, see Algorithm~\ref{alg:W}.
\begin{algorithm}[h!t]
\caption{Procedure \texttt{IsWTautology}($\varphi$)}
\label{alg:W}
\begin{algorithmic}[1] 
\State \textbf{procedure} \texttt{IsWTautology($\varphi$)}
    \State \mbox{\Comment\textbf{Input:} formula $\varphi$.}\hfill {~}
    \State \mbox{\Comment\textbf{Output:} $\mathtt{true}$ or $\mathtt{false}$.}\hfill {~}

    \If{$\varphi\in\Pi$}
        \Return \texttt{false};
    \ElsIf{$\varphi=\alpha\wedge\beta$}
        \Return
               \texttt{IsWTautology($\alpha$)}\ \texttt{\&}
               \texttt{IsWTautology($\beta$)};
    \ElsIf{$\varphi=\alpha\to\beta$}
        \For {{\bf each} $\gamma$ \textbf{in} $K(\beta)$}
            \If {{\bf not} \texttt{IsWTautology($\gamma$)}}
               \If {{\bf not} $\gamma\in K(\alpha)$} \Return \texttt{false};
               \EndIf
            \EndIf
        \EndFor
    \EndIf
    \State \Return \texttt{true};
\end{algorithmic}
\end{algorithm}

Thus,
\begin{equation}
\begin{array}{lcl}
\varphi\in\logic{W} & \iff & \mathtt{IsWTautology(\varphi)} = \mathtt{true}.
\end{array}
\label{eq:W:IsWT}
\end{equation}

\begin{imtheorem}
\label{th:W:decidability}
\textit{Set\/ $\logic{W}$ is decidable.}
\end{imtheorem}

\begin{proof}
Follows from~\eqref{eq:W:IsWT}.
\end{proof}

In fact, the presented algorithm gives not only the decidability of the set of tautologies of the minimal well-determined logic, but also allows us to obtain an upper bound for the complexity of the decision problem for this set.
To simplify the procedure for obtaining a complexity bound, we modify the algorithm described above: when computing the value of a subformula $\psi$ of the formula $\varphi$, we save this value as $\mathtt{WT[\mbox{$\psi$}]}$, in order to avoid repeated computations of the same thing. We assume that before the computation begins, $\mathtt{WT[\mbox{$\psi$}]}=\mathtt{NULL}$ for each subformula $\psi$ of the tested formula $\varphi$, i.e., the value $\mathtt{WT[\mbox{$\psi$}]}$ is undefined. The value $\mathtt{WT[\mbox{$\varphi$}]}$ is computed by the procedure $\mathtt{WTautology(\mbox{$x$})}$, which takes an arbitrary formula $\varphi$ as $x$; this procedure does not output anything (i.e., outputs $\mathtt{NULL}$), because the important thing is not the output of the procedure, but the value $\mathtt{WT[\mbox{$\varphi$}]}$. At the beginning of the computation, the procedure checks whether this value has already been computed, and only if it has not been computed, computes it recursively.

A formal description of the procedure $\mathtt{WTautology}(\varphi)$ is presented in pseudocode, see~Algorithm~\ref{alg:WR}.

\begin{algorithm}[h!t]
\caption{Procedure \texttt{WTautology}($\varphi$)}
\label{alg:WR}
\begin{algorithmic}[1]
\State \textbf{procedure} \texttt{WTautology($\varphi$)}
    \State \mbox{\Comment\textbf{Input:} formula $\varphi$.}\hfill {~}
    \State \mbox{\Comment\textbf{Output:} $\mathtt{NULL}$.}\hfill {~}
    \State \mbox{\Comment\textbf{Before computation:} \texttt{WT[\mbox{$\psi$}] $=$ NULL} for all $\psi$.}\hfill {~}

    \If{\texttt{WT[\mbox{$\varphi$}] $=$ NULL}}
        \If{$\varphi\in\Pi$}
            \texttt{WT[\mbox{$\varphi$}] $:=$ \texttt{false}};
        \ElsIf{$\varphi=\alpha\wedge\beta$}
            \State \texttt{WTautology($\alpha$)};
                   \Comment{computes \texttt{WT($\alpha$)} if necessary}
            \State \texttt{WTautology($\beta$)};
                   \Comment{computes \texttt{WT($\beta$)} if necessary}
            \State
            \texttt{WT[\mbox{$\varphi$}]} $:=$
                   \texttt{WT[$\alpha$]}\ \texttt{\&}
                   \texttt{WT[$\beta$]};
        \ElsIf{$\varphi=\alpha\to\beta$}
            \State \texttt{WT[\mbox{$\varphi$}] $:=$ \texttt{true}};
            \For {{\bf each} $\gamma$ \textbf{in} $K(\beta)$}
                \State \texttt{WTautology($\gamma$)};
                       \Comment{computes \texttt{WT($\gamma$)} if necessary}
                \If {{\bf not} \texttt{WT[$\gamma$]}}
                   \If {{\bf not} $\gamma\in K(\alpha)$}
                       \State \texttt{WT[\mbox{$\varphi$}] $:=$ \texttt{false}};
                       \State \textbf{break};
                   \EndIf
                \EndIf
            \EndFor
        \EndIf
    \EndIf
\State \Return \texttt{NULL};
\end{algorithmic}
\end{algorithm}

Let us give some explanations to this description.
If the condition in line~5 is not satisfied, i.e., the value $\mathtt{WT[\mbox{$\varphi$}]}$ has already been computed earlier, then no actions are performed. Next, the structure of the procedure $\mathtt{IsWTautology}(\varphi)$ is repeated, with the difference that instead of the value $\mathtt{IsWTautology}(\psi)$, we use the value $\mathtt{WT[\mbox{$\psi$}]}$, which is guaranteed to be defined after the call $\mathtt{WTautology}(\psi)$: these are lines 8, 9 and~14, where $\psi$ is taken to be $\alpha$, $\beta$ or $\gamma$, respectively. When computing $\mathtt{WT[\mbox{$\varphi$}]}$ in the case $\varphi=\alpha\to\beta$, we initially set $\mathtt{WT[\mbox{$\varphi$}]} = \mathtt{true}$, but then check a condition that may lead to changing this value; if the value changes, further checking is not required and it stops at line~18. Line~19 becomes merely a formality, simply ending the procedure.

\begin{imtheorem}
\label{th:W:p:decidability}
\textit{Set\/ $\logic{W}$ is decidable in polynomial time.}
\end{imtheorem}

\begin{proof}
The value of each subformula of the formula $\varphi$ is computed once, but when computing the value of an implication $\alpha\to\beta$, it is necessary to iterate over all elements of the set $K(\beta)$, which is handled by line~13. The number of steps in this iteration is bounded above by the number of elements in $K(\beta)$, which, in turn, does not exceed the length of the formula~$\varphi$. Thus, the total number of steps in computing $\mathtt{WT[\mbox{$\varphi$}]}$ is bounded above by a quadratic function in the length of~$\varphi$.
\end{proof}

\section{Conclusion}
   \setcounter{equation}{0}
   \label{sec:8}

We have shown that the class of lower semilattices with a greatest element can be taken as a semantics for the minimal well-determined logic. In algebras of this class there is no operation corresponding to implication, and to evaluate formulas with implication we use an extension of the notion of a valuation.
This makes it possible to pose various questions related to semantics for this logic, such as the question of the finite model property 
in the class of such algebras. 

The proposed semantics can be regarded as a basis for investigating semantics of extensions of the minimal well-determined logic.
Since $(\mathit{MP})$ is admissible for the minimal well-determined logic, among its extensions a special interest is the logic \insMR{$W+(\mathit{MP})$} obtained from the minimal well-determined logic by requiring that not only the logic itself but all its theories be closed under~$(\mathit{MP})$. 
Note that the logic obtained in this way is not well-determined.

As has been shown, the set $\logic{W}$ of tautologies of the minimal well-determined logic is decidable; moreover, it is decidable in polynomial time. This polynomial-time decidability result is perhaps unexpected, yet it admits a simple explanation given by Lemma~\ref{l4}.
\insMR{Since $\logic{W}$ is also the set of tautologies of \insMR{$W+(\mathit{MP})$}, we obtain an example of polynomial-time decidable logic closed under~$(\mathit{MP})$.}

Note that quite often the sets of tautologies of propositional logics contain the set of classical tautologies as their natural fragment: for example, $\logic{Cl}$ is the modal-free fragment of modal logics~\cite{IG:MR:Feys1965,IG:MR:ChZ}, and also embeds into superintuitionistic logics via various translations~\cite{IG:MR:ChZ,IG:MR:FerreiraOliva:2010}. Nevertheless, this is not the case for $\logic{W}$, since $\logic{Cl}$ is $\ccls{coNP}$-complete~\cite{IG:MR:cook1971,IG:MR:levin1973}. No essential changes occur if we restrict ourselves to the conjunctive-implicative fragment of classical logic, i.e., the fragment with the same set of connectives as in the language of the minimal well-determined logic: it is easy to show that the problem of non-membership of formulas in the set $\logic{Cl}$ in the language with conjunction and implication remains $\ccls{NP}$-complete, i.e., the following statement holds.

\begin{improposition}
\textit{The conjunctive-implicative fragment of\/ $\logic{Cl}$ is\/ $\ccls{coNP}$-complete.}
\end{improposition}

\begin{proof}
It is known that the satisfiability problem for formulas in conjunctive normal form, where each conjunctive clause is a disjunction of three literals, is $\ccls{coNP}$-complete, which follows from the Cook--Levin theorem~\cite{IG:MR:cook1971,IG:MR:levin1973} and Tseytin's construction~\cite{IG:MR:tseitin1968}.
Let $\varphi$ be such a formula, i.e.,
$$
\begin{array}{lcl}
\varphi
  & =
  & \displaystyle
    \bigwedge\limits_{\mathclap{k=1}}^m(l^k_1\vee l^k_2\vee l^k_3),
\end{array}
$$
where each $l^k_i$ is a literal, i.e., either a variable or the negation of a variable.
Let $q_\varphi$ be a fixed variable not occurring in $\varphi$. Set
$$
\begin{array}{lcl}
\lambda^k_i
  & =
  & \left\{
    \begin{array}{ll}
      p      & \text{if $l^k_i=\neg p$;}\smallskip\\
      p\to q_\varphi & \text{if $l^k_i= p$,}\\
    \end{array}
    \right.
\end{array}
$$
and also
$$
\begin{array}{lcl}
\varphi^\ast
  & =
  & \displaystyle
    \bigwedge\limits_{\mathclap{k=1}}^m(\lambda^k_1\wedge \lambda^k_2\wedge \lambda^k_3 \to q_\varphi)\to q_\varphi.
\end{array}
$$
It is easy to see that the following equivalence holds:
$$
\begin{array}{lcl}
\text{$\varphi$ is satisfiable}
  & \iff
  & \text{$\varphi^\ast\not\in\logic{Cl}$.}
\end{array}
$$
Indeed, it suffices to note that substituting any unsatisfiable formula for $q_\varphi$ in $\varphi^\ast$ yields a formula equivalent to $\neg\varphi$ in classical logic, and that $\varphi^\ast$ can be refuted only by valuations that falsify~$q_\varphi$.

Since $\varphi^\ast$ is constructed from $\varphi$ in polynomial time, we obtain that the conjunctive-implicative fragment of the set $\logic{Cl}$ is $\ccls{coNP}$-hard. The membership of this fragment in the class $\ccls{coNP}$ follows from the membership of the set of classical tautologies in the full language in this class.
\end{proof}

Assuming that $\ccls{P}\ne\ccls{NP}$ (see~\cite{IG:MR:Jaffe:2000}), this observation implies that even the conjunctive-implicative fragment of the set of tautologies of classical logic cannot be embedded into $\logic{W}$ by any polynomial-time algorithm. Nevertheless, such embeddings exist for finite-variable fragments of the set $\logic{Cl}$, since each such fragment is decidable in polynomial time.

Note that, unlike classical logic, the sets of tautologies of many non-classical logics are polynomial-time embeddable into their finite-variable fragments: this is true for monomodal systems---both normal~\cite{IG:MR:spaan1993, IG:MR:halpern1995, IG:MR:2002:LI:1, IG:MR:chagrov2003, IG:MR:svejdar2003, IG:MR:agadzhanian2022, IG:MR:rybakov2025} and non-normal~\cite{IG:kudinov2025, IG:MR:2025:Nsk:2}---as well as for polymodal~\cite{IG:MR:2007:Tver, IG:MR:pahomov2014, IG:MR:2018:IGPL, IG:MR:2022:JLC, IG:MR:2022:TCS}, superintuitionistic~\cite{IG:MR:2004:LI, IG:MR:2008:JANCL}, and some other logics~\cite{IG:MR:2003:LI, IG:MR:2025:JLC:embedding, IG:MR:2024:Nsk:1, IG:MR:onoprienko-rybakov-submitted}. It would be interesting to understand how to construct such embeddings in the case of~$\logic{W}$. Note that the polynomial-time decidability of the set $\logic{W}$ implies that $\logic{W}$ is polynomial-time reducible to any of its finite-variable fragments.

\begin{improposition}
\textit{The set $\logic{W}$ is polynomial-time reducible to its one-variable fragment.}
\end{improposition}

\begin{proof}
Define a function $f$ that maps each conjunctive-implicative formula to a formula in one variable by setting
$$
\begin{array}{lcl}
f(\varphi)
  & =
  & \left\{
    \begin{array}{ll}
    p\to p & \text{if $\varphi\in \logic{W}$;} \\
    p      & \text{if $\varphi\not\in \logic{W}$.}
    \end{array}
    \right.
\end{array}
$$
Then $f$ is a polynomial-time reduction from $\logic{W}$ to its one-variable fragment by Theorem~\ref{th:W:p:decidability}.
\end{proof}

Nevertheless, it would be interesting to investigate whether such reductions based on structure-preserving translations exist.  Indeed, the existing literature presents a mixed picture. On the one hand, the works cited above construct explicit polynomial-time embeddings of numerous logics into their finite-variable fragments. On the other hand, for a broad family of logics---including $\logic{S4}$, $\logic{K4}$, $\logic{GL}$, $\logic{Grz}$, $\logic{S4.3}$, $\logic{K4.3}$, $\logic{GL.3}$, $\logic{Grz.3}$, $\logic{E}$, $\logic{EM}$, $\logic{EN}$, $\logic{EMN}$, and many others---results about their complexity imply that polynomial-time embeddings into the one-variable fragments (even into the variable-free fragments for some) must exist, yet these proofs do not yield explicit constructions, and in particular, no structure-preserving translation is provided. How to construct such embeddings explicitly (and whether it is possible) remains unclear to the authors. Thus, the inquiry concerning $\logic{W}$ fits naturally into this broader picture.

\subsection*{Acknowledgements}
\addcontentsline{toc}{section}{Acknowledgements}

During the preparation of this work, the language model DeepSeek (DeepSeek-R1) was used to check the correctness of proofs, to formalize the decision procedure for the set $\logic{W}$, and to check English.

\subsection*{Funding}
\addcontentsline{toc}{section}{Funding}

The research leading to these results has received funding from the Basic Research Program at the National Research University Higher School of Economics for the first author.
The research is supported by MSHE RF GZ project for the second author.

\appendix
\section{Some technical proofs}
   \setcounter{equation}{0}
   \label{sec:A}
   \label{sec:Pr}

\insMR{
We present proofs of some facts to which we referred in the main text. For convenience, we introduce the following notation:
$$
\begin{array}{lcl}
\varphi_1 & = & p\to p; \\
\varphi_2 & = & p\wedge q\to p; \\
\varphi_3 & = & p\wedge q\to q\wedge p; \\
\varphi_4 & = & (p\wedge q)\wedge r\to p\wedge (q\wedge r); \\
\varphi_5 & = & p\wedge (q\wedge r)\to (p\wedge q)\wedge r.
\end{array}
$$
Let also $\mathcal{B}=\set{\varphi_1, \varphi_2, \varphi_3, \varphi_4, \varphi_5}$.

\begin{imlemma} \label{1:IG}
The formulas of\/ $\mathcal{B}$ are derivable from $\mathcal{A}$ using $(\mathit{TR})$, $(\mathit{CM})$, and $(\mathit{EA})$.
\end{imlemma}

\begin{proof}
We show for each formula in $\mathcal{B}$, how it is derived from $\mathcal{A}$ using $(\mathit{TR})$, $(\mathit{CM})$, and $(\mathit{EA})$.
\begin{itemize}[leftmargin=4em, noitemsep, topsep=5pt, parsep=1pt]
\item[$\varphi_1$:]  
The formulas $p\to p\wedge p$ and $p\wedge p\to p$ are obtained by $\mathit{Ax}1$ and $\mathit{Ax}2$, respectively, and from these formulas by $(\mathit{TR})$ we obtain~$\varphi_1$.

\item[$\varphi_2$:]  
From $\varphi_1$ by $(\mathit{EA})$ we obtain $\varphi_2$.

\item[$\varphi_3$:]  
Applying $(\mathit{CM})$ to the formulas $p\wedge q\to q$ (obtained by $\mathit{Ax}2$) and $\varphi_2$, we obtain $(p\wedge q)\wedge(p\wedge q)\to q\wedge p$. From the latter formula and the formula $p\wedge q\to (p\wedge q) \wedge (p\wedge q)$ (obtained by $\mathit{Ax}1$) by $(\mathit{TR})$ we obtain $\varphi_3$. 

\item[$\varphi_4$:]  
Let $\omega = (p\wedge q)\wedge r$. Set
$$
\begin{array}{lcl}
\begin{array}{lcl}
\omega_1 & = & \omega\to r; \\
\omega_2 & = & \omega\to(p\wedge q); \\
\omega_3 & = & \omega\to q; \\
\omega_4 & = & \omega\to p; \\
\end{array}
& {}\quad{} &
\begin{array}{lcl}
\omega_5 & = & \omega\wedge\omega\to q\wedge r; \\
\omega_6 & = & \omega\to q\wedge r; \\
\omega_7 & = & \omega\wedge\omega\to p\wedge (q\wedge r); \\
\omega_8 & = & \omega\to\omega\wedge\omega.
\end{array}
\end{array}
$$
The formula $\omega_8$ is obtained by $\mathit{Ax}1$, the formula $\omega_1$ is obtained by $\mathit{Ax}2$, the formula $\omega_2$ is a substitution instance of $\varphi_2$, the formulas $\omega_3$ and $\omega_4$ are obtained from $\omega_2$, $\varphi_2$ and $p\wedge q\to q$ using $(\mathit{TR})$. From $\omega_1$ and $\omega_3$ by $(\mathit{CM})$ we obtain~$\omega_5$. Applying $(\mathit{TR})$ to $\omega_5$ and $\omega_8$ we obtain $\omega_6$. From $\omega_4$ and $\omega_6$ by $(\mathit{CM})$ we obtain~$\omega_7$. From $\omega_7$ and $\omega_8$ by $(\mathit{TR})$ we obtain~$\varphi_4$. 

\item[$\varphi_5$:]  
The formula $\varphi_5$ is derived similarly.
\end{itemize}
Thus we have obtained the required.
\end{proof}

\begin{imlemma} \label{2:IG}
Any deductive set of formulas satisfies $(C1)$, $(C2)$, and $(C3)$.
\end{imlemma}

\begin{proof}
Let $L$ be a deductive set. Then there exists a well-determined logic $C$ such that $C(\varnothing)=L$. Using $C$, we show that $L$ satisfies $(C1)$--$(C3)$.
\begin{itemize}[leftmargin=4em, noitemsep, topsep=5pt, parsep=1pt]
\item[$(C1)$:]
As the set of tautologies of a standard consequence, $L$ is invariant \cite[p.\,11, 2.7]{IG:MR:1}, which yields~$(C1)$.
\item[$(C2)$:]
By $(A1)$--$(A3)$ and $(B2)$ we obtain $p\in C(p)=C(p\wedge p)$ and $q\in C(p, q)=C(p\wedge q)$, which by $(B1)$ gives~$(C2)$.
\item[$(C3)$:]
Direct verification of the closure of $L$ under the rules in $\lang{R}$.
For example, we prove that $L$ is closed under $(\mathit{TR})$. Suppose $\alpha\to \beta\in L$ and $\beta\to \gamma\in L$. Then by $(B1)$ we obtain $\beta\in C(\alpha)$ and $\gamma\in C(\beta)$. By the properties of the consequence operation $C$, we have $C(\beta)\subseteq C(\alpha)$, hence $\gamma\in C(\alpha)$, and therefore $\alpha\to\gamma\in L$. The verification of the remaining rules is similar and is left to the reader.
\end{itemize}
Thus, $(C1)$--$(C3)$ hold for $L$.
\end{proof}

\begin{imlemma} \label{3:IG}
Any set of formulas satisfying $(C1)$, $(C2)$, and $(C3)$ is deductive.
\end{imlemma}

\begin{proof}
Let $L$ be a set of formulas satisfying $(C1)$--$(C3)$. We show that $\vec{L}$ is the required consequence.

First, we prove $(A1)$--$(A5)$ for $L$. Observe that, due to $(C1)$--$(C3)$, the conditions of Lemma~\ref{1:IG} hold for~$L$. From Lemma~\ref{1:IG} we obtain Lemma~\ref{l1} (see the proof of Lemma~\ref{l1}), which yields \eqref{eq:8a} for~$\vec L$.
\begin{itemize}[leftmargin=4em, noitemsep, topsep=5pt, parsep=1pt]
\item[$(A1)$:]  
From Lemma~\ref{1:IG} it follows that $\alpha\to \alpha\in L$ for any formula $\alpha$, and hence $\alpha\in{\vec L}(\alpha)$ by \eqref{eq:8a}, which gives~$(A1)$. 

\item[$(A2)$:]  
Immediately follows from \eqref{eq:8a}.

\item[$(A3)$:]  
The inclusion $\vec L(X)\subseteq\vec L(\vec L(X))$ follows from $(A1)$, and we only need to prove the reverse inclusion.
Suppose $\alpha\in \vec L(\vec L(X))$. Then, by \eqref{eq:8a}, there exists
$\Gamma\in \mathcal{P}_{\mathit{fin}}^+(L\cup\vec L(X))$ such that ${\bigwedge}\Gamma\to\alpha \in L$. 
Note that $\vec L(L)=\vec L(\varnothing)$ by \eqref{eq:8a}, and then $L\subseteq\vec L(\varnothing)$ by $(A1)$, whence we obtain $L\cup\vec L(X)=\vec L(X)$.
Thus $\Gamma\subseteq\vec L(X)$. 
Let $\Gamma=\{\alpha_1, \ldots, \alpha_n\}$. 
Then, by \eqref{eq:8a}, for each $i\in\set{1,\ldots,n}$ there exists $\Gamma_i\in \mathcal{P}_{\mathit{fin}}^+(L\cup X)$ such that 
${\bigwedge}\Gamma_{i}\to\alpha_i\in L$.
Let $\beta_i={\bigwedge}\Gamma_{i}$, where $i\in\set{1,\ldots,n}$. 
Set
$$
\begin{array}{lcl}
\begin{array}{lcl}
\alpha^1 & = & \alpha_1; \smallskip\\
\alpha^{k+1} & = & \alpha^k\wedge\alpha_{k+1};
\end{array}
& {}\quad{} &
\begin{array}{lcl}
\beta^1 & = & \beta_1; \smallskip\\
\beta^{k+1} & = & \beta^k\wedge\beta_{k+1},
\end{array}
\end{array}
$$
where $k\in\set{1,\ldots,n-1}$.
For each $k\in\set{1,\ldots,n-1}$, from the formulas $\beta^k\to\alpha^k$ and $\beta_{k+1}\to\alpha_{k+1}$ by $(\mathit{CM})$ we obtain $\beta^{k+1}\to\alpha^{k+1}$.
Applying $(\mathit{TR})$ to $\beta^n\to\alpha^n$ and $\alpha^n\to\alpha$, we obtain $\beta^n\to\alpha$. 
Since $\bigcup\set{\Gamma_i : i\in\set{1,\ldots,n}}\subseteq X\cup L$, by \eqref{eq:8a} we obtain $\alpha\in\vec L(X)$. Hence,
$\vec L(\vec L(X))\subseteq\vec L(X)$, and therefore $\vec L(X)=\vec L(\vec L(X))$.

\item[$(A4)$:]  
Follows from \eqref{eq:8a} and $(C1)$.

\item[$(A5)$:]  
Follows from \eqref{eq:8a}.
\end{itemize}

We show that $L=\vec L(\varnothing)$. 
The inclusion $L\subseteq\vec L(\varnothing)$ is justified above, and we need to prove the reverse inclusion. Suppose $\alpha\in\vec L(\varnothing)$. Then there exist formulas $\alpha_1, \ldots, \alpha_n\in L$ such that $\alpha_1\wedge \ldots \wedge\alpha_n\to\alpha\in L$.
Applying $(\mathit{AD})$ we obtain $\alpha_1\wedge \ldots \wedge\alpha_n\in L$, and then by $(\mathit{MP})$ we get $\alpha\in L$.

It is easy to see that $\vec L$ satisfies the implication $(\Rightarrow)$ in \eqref{eq:WD}: if $[\Gamma\to\alpha]^\wedge\subseteq \vec{L}(\varnothing)$ for some nonempty finite $\Gamma$, then $\alpha\in \vec{L}(\Gamma)$ by \eqref{eq:8}.

We show that the implication $(\Leftarrow)$ in \eqref{eq:WD} holds for~$\vec L$.
Suppose $\alpha\in \vec L(\Gamma)$ for some nonempty finite~$\Gamma$. By \eqref{eq:8a}, there exists $\Delta\in \mathcal{P}_{\mathit{fin}}^+(\Gamma\cup L)$ such that 
${\bigwedge}\Delta\to\alpha\in L$. Using $(\mathit{EA})$, we add to the antecedent as conjuncts the formulas from $\Gamma$ that are missing in it, and using $(\mathit{CV})$, we remove from it the conjuncts that lie in $L\setminus\Gamma$; we obtain ${\bigwedge}\Gamma\to\alpha\in L$. Then, by \eqref{eq2:lem:l1:IG}, $[\Gamma\to\alpha]^\wedge\subseteq L$.
\end{proof}

On the set of standard consequences, we introduce a binary relation $\leqslant$ by setting, for any consequences $C_1$ and $C_2$,
$$
\begin{array}{lcl}
C_1\leqslant C_2 & \bydef & \text{for every set of formulas $X$, we have $C_1(X)\subseteq C_2(X)$.}
\end{array}
$$
It is easy to see that $\leqslant$ is a partial order.

\begin{theorem}  \label{4:IG}
Let $L$ be a deductive set and $C$ a well-determined logic such that $C(\varnothing)=L$. Then $C=\vec L$. 
\end{theorem}

\begin{proof}
Suppose $C\not=\vec L$. Then either $C\not\leqslant\vec L$ or $\vec L\not\leqslant C$.

Assume $C\not\leqslant\vec L$. Then there exist $X$ and $\alpha$ such that $\alpha\in C(X)$ and $\alpha\not\in\vec L(X)$. By $(A5)$ for $C$, there exists a finite $\Gamma\subseteq X$ such that $\alpha\in C(\Gamma)$. Since $C$ is well-determined, ${\bigwedge}\Gamma\to\alpha\in C(\varnothing)=L$. Then, by \eqref{eq:8a}, $\alpha\in\vec L(\Gamma)$, and hence $\alpha\not\in\vec L(X)$. Contradiction.

Assume $\vec L\not\leqslant C$. Then there exist $X$ and $\alpha$ such that $\alpha\not\in C(X)$ and $\alpha\in\vec L(X)$. 
From the latter it follows that there exists $\Gamma\in\mathcal{P}_{\mathit{fin}}^+(X\cup L)$ such that ${\bigwedge}\Gamma\to\alpha\in L$. 
Let $\Delta = \Gamma\setminus L$. Note that $\Delta\ne\varnothing$, since otherwise we would have $\alpha\in L$, and hence $\alpha\in C(X)$, which is not the case. Using $(\mathit{CV})$, we remove from the antecedent of ${\bigwedge}\Gamma\to\alpha$ the conjuncts not belonging to $\Delta$; we obtain ${\bigwedge}\Delta\to\alpha\in L$. By $(B1)$ and $(B2)$ for $C$, we get $\alpha\in C(\Delta)$, and therefore $\alpha\in C(X)$. Contradiction.

Thus $C=\vec L$.
\end{proof}
}


\begin{thebibliography}{99}
\addcontentsline{toc}{section}{References}

\bibitem{IG:MR:agadzhanian2022}
I.~Agadzhanian, M.~Rybakov.
\newblock Complexity of the variable-free fragment of the weak Grzegorczyk logic.
\newblock arXiv:2211.14571, 2022, 13\,p.

\bibitem{IG:MR:chagrov2003}
A.~Chagrov, M.~Rybakov.
\newblock How many variables does one need to prove PSPACE-hardness of modal logics?
\newblock Advances in Modal Logic, 4:71--82, 2003.

\bibitem{IG:MR:ChZ}
A.~Chagrov, M.~Zakharyaschev.
\newblock Modal Logic.
\newblock Oxford University Press, 1997, 605\,p.

\bibitem{IG:MR:cook1971}
S.\,A.~Cook.
\newblock The Complexity of Theorem-Proving Procedures.
\newblock In: Proceedings of the Third Annual ACM Symposium on the Theory of Computation, 1971, pp.~151--158.

\bibitem{IG:MR:FerreiraOliva:2010}
G.~Ferreira, P.~Oliva.
\newblock On various negative translations.
\newblock In: Proceedings Third International Workshop on Classical Logic and Computation CL\&C~2010, Brno, Czech Republic, 2010, pp.~21--33.

\bibitem{IG:MR:Feys1965}
R.~Feys.
\newblock Modal Logics.
\newblock E.~Nauwelaerts, Louvain, 1965, 219\,p.

\bibitem{IG:MR:FJP:2003}
J.\,M.~Font, R.~Jansana, D.~Pigozzi.
\newblock A Survey of Abstract Algebraic Logic.
\newblock Studia Logica, 74(1):13--97, 2003.

\bibitem{IG:MR:3}
I.\,A.~Gorbunov.
\newblock Well-defined logics.
\newblock Logical Investigations, 17:95--108, 2011.
\newblock (In~Russian)

\bibitem{IG:MR:4}
I.\,A.~Gorbunov.
\newblock An effective criterion of deductivity of sets of logical formulas.
\newblock Herald of Tver State University. Series: Applied Mathematics, 1:95--103, 2017.
\newblock (In~Russian)

\bibitem{IG:MR:halpern1995}
J.\,Y.~Halpern.
\newblock The effect of bounding the number of primitive propositions and the depth of nesting on the complexity of modal logic.
\newblock Artificial Intelligence, 75(2):361--372, 1995.

\bibitem{IG:MR:Jaffe:2000}
A.\,M.~Jaffe.
\newblock The millennium grand challenge in mathematics.
\newblock Notices of the American Mathematical Society, 53(6):652--660, 2000.

\bibitem{IG:kudinov2025}
A.~Kudinov, M.~Rybakov.
\newblock Complexity of the variable-free fragments of non-normal modal logics (extended version).
\newblock arXiv:2507.09136, 2025, 8\,p.

\bibitem{IG:MR:levin1973}
L.\,A.~Levin.
\newblock Universal enumeration problems.
\newblock Problems of Information Transmission, 9(3):115--116, 1973.
\newblock (In~Russian)

\bibitem{IG:MR:onoprienko-rybakov-submitted}
A.~Onoprienko, M.~Rybakov.
\newblock Joint logic of problems and propositions: translations, embeddings, and computational complexity.
\newblock Submitted.

\bibitem{IG:MR:2024:Nsk:1}
A.\,A.~Onoprienko, M.\,N.~Rybakov.
\newblock Complexity of the logic HC with one variable.
\newblock In: International Conference ``Maltsev Meetings''. Abstracts of talks, Novosibirsk, 2024, p.~44.
\newblock (In~Russian)

\bibitem{IG:MR:pahomov2014}
F.~Pahomov.
\newblock On the complexity of the closed fragment of Japaridze's provability logic.
\newblock Archive for Mathematical Logic, 53:949--967, 2014.

\bibitem{IG:MR:2002:LI:1}
M.\,N.~Rybakov, A.\,V.~Chagrov.
\newblock Constant formulas in modal logics: the decidability problem.
\newblock Logical Investigations, 9:202--220, 2002.
\newblock (In~Russian)

\bibitem{IG:MR:2003:LI}
M.\,N.~Rybakov.
\newblock Complexity of the decidability problem for the basic and formal logics.
\newblock Logical Investigations, 10:158--166, 2003.
\newblock (In~Russian)

\bibitem{IG:MR:2004:LI}
M.\,N.~Rybakov.
\newblock Embedding of intuitionistic logic into its two-variable fragment and the complexity of this fragment.
\newblock Logical Investigations, 11:247--261, 2004.
\newblock (In~Russian)

\bibitem{IG:MR:2006:AiML}
M.~Rybakov.
\newblock Complexity of intuitionistic and Visser's basic and formal logics in finitely many variables.
\newblock In: Advances in Modal Logic, vol.~6, College Publications, London, 2006, pp.~393--411.

\bibitem{IG:MR:2007:Tver}
M.\,N.~Rybakov.
\newblock Complexity of the constant fragment of propositional dynamic logic.
\newblock Herald of Tver State University. Series: Applied Mathematics, 5:5--17, 2007.
\newblock (In~Russian)

\bibitem{IG:MR:rybakov2007}
M.~Rybakov.
\newblock Complexity of finite-variable fragments of EXPTIME-complete logics.
\newblock Journal of Applied Non-Classical Logics, 17(3):359--382, 2007.

\bibitem{IG:MR:2008:JANCL}
M.~Rybakov.
\newblock Complexity of intuitionistic propositional logic and its fragments.
\newblock Journal of Applied Non-Classical Logics, 18(2--3):267--292, 2008.

\bibitem{IG:MR:2018:IGPL}
M.~Rybakov, D.~Shkatov.
\newblock Complexity and expressivity of propositional dynamic logics with finitely many variables.
\newblock Logic Journal of the IGPL, 26(5):539--547, 2018.

\bibitem{IG:MR:2022:JLC}
M.~Rybakov, D.~Shkatov.
\newblock Complexity of finite-variable fragments of products with non-transitive modal logics.
\newblock Journal of Logic and Computation, 32(5):853--870, 2022.

\bibitem{IG:MR:2022:TCS}
M.~Rybakov, D.~Shkatov.
\newblock Complexity of finite-variable fragments of propositional temporal and modal logics of computation.
\newblock Theoretical Computer Science, 925:45--60, 2022.

\bibitem{IG:MR:rybakov2025}
M.~Rybakov, M.~Shcherbakov.
\newblock Logics with the axiom of convergence: complexity with a small number of variables in the language (extended version).
\newblock arXiv:2507.12343, 2025, 6\,p.

\bibitem{IG:MR:2025:JLC:embedding}
M.~Rybakov, D.~Shkatov.
\newblock Polytime embedding of intuitionistic modal logics into their one-variable fragments.
\newblock Journal of Logic and Computation, 35(4):Article exae077, 2025.

\bibitem{IG:MR:2025:Nsk:2}
M.\,N.~Rybakov.
\newblock Complexity of the logics S2 and S3.
\newblock In: International Conference ``Maltsev Meetings''. Abstracts of talks, Novosibirsk, 2025, p.~100.
\newblock (In~Russian)

\bibitem{IG:MR:spaan1993}
E.~Spaan.
\newblock Complexity of Modal Logics.
\newblock PhD thesis, Universiteit van Amsterdam, 1993.

\bibitem{IG:MR:svejdar2003}
V.~\v{S}vejdar.
\newblock The decision problem of provability logic with only one atom.
\newblock Archive for Mathematical Logic, 42(8):763--768, 2003.

\bibitem{IG:MR:tseitin1968}
G.\,S.~Tseytin.
\newblock On the complexity of derivation in propositional calculus.
\newblock In: Studies in Constructive Mathematics and Mathematical Logic.~II, Notes of Scientific Seminars of LOMI, 8:234--259, 1968.
\newblock (In~Russian)

\bibitem{IG:MR:1}
R.~W\'{o}jcicki.
\newblock Lectures on Propositional Calculi.
\newblock Ossolineum, Wroclaw, 1984, 292\,p.

\bibitem{IG:MR:2}
R.~W\'{o}jcicki.
\newblock Theory of Logical Calculi: Basic Theory of Consequence Operations.
\newblock Kluwer Academic Publishers, 1988, 473\,p.

\end{thebibliography}
\end{document}